\documentclass[11pt]{paper}
\usepackage[T1]{fontenc}
\usepackage{amsmath,amsthm}
\usepackage{amsfonts}  
\usepackage{amssymb}
\usepackage{graphicx} 
\usepackage{bm}
\usepackage{hyperref}
\usepackage{xcolor}
\usepackage{nicefrac}

\usefont{T1}{cmr}{m}{scit}
\usepackage[normalem]{ulem}
\newcounter{corr}
\definecolor{violet}{rgb}{0.580,0.,0.827}
\newcommand{\corr}[3]{\typeout{Warning : a correction remains in page
\thepage}
				\stepcounter{corr}        
				{\color{blue}\ifmmode\text{\,\sout{\ensuremath{#1}}\,}\else\sout{#1}\fi}
       {\color{red}#2}
       {\color{violet} #3}}
\numberwithin{equation}{section}
\usepackage{constants}
\newconstantfamily{ctrcst}{symbol=C,reset={bibunit}}
\def\ctel#1{\ensuremath{\Cl[ctrcst]{#1}}}
\def\cter#1{\ensuremath{\Cr{#1}}}

 \newtheorem{thm}{Theorem}[section]
 \newtheorem{lem}[thm]{Lemma}
 \newtheorem{remark}[thm]{Remark}
 \newtheorem{definition}[thm]{Definition}

 \def\Id{\mathop{\rm Id}\nolimits}
 
 \def\dist{\mathop{\rm \delta}\nolimits}

\newcommand{\sym}{{\rm s}}
\newcommand{\asym}{{\rm a}}

\newcommand{\dsp}{\displaystyle} 
\newcommand{\mathbi}[1]{{\boldsymbol #1}}

\newcommand{\bv}{\mathbi{v}}

\newcommand{\bF}{\mathbi{F}}

\newcommand{\bvarphi}{\mathbi{\varphi}}

\newcommand{\disc}{{h}}

\renewcommand{\div}{{\mathop{\rm div}}}

\newcommand{\wunp}{{H_{\agrad}}}

\newcommand{\lp}{L}

\newcommand{\lpd}{\mathbi{\lp}}
\newcommand{\hdiv}{\mathbi{H}_\adiv}

\newcommand{\adiv}{\textsc{D}}
\newcommand{\agrad}{\textsc{\bf G}}
\newcommand{\api}{\textsc{P}}

\newcommand{\wn}{{\mathfrak{S}}}

\newcommand{\N}{\mathbb N}

\renewcommand{\phi}{\varphi}

\newcommand{\R}{\mathbb R}

\newcommand{\ie}{\emph{i.e.\/}}

\newcommand{\norm}[2]{\left\Vert#1\right\Vert_{#2}}
\newcommand{\consist}{\sigma}
\newcommand{\limconf}{\zeta}
\newcommand{\disT}{\delta^{(T)}}
\newcommand{\conssT}{\widehat{\sigma}^{(T)}}
\newcommand{\limcT}{\zeta^{(T)}}

\newcommand{\ba}{\begin{array}{llll}   }
\newcommand{\ea}{\end{array}}

\newcommand{\be}{\begin{equation}}
\newcommand{\ee}{\end{equation}}

\begin{document}	

\title{Optimal error estimates for non-conforming approximations of linear parabolic problems with minimal regularity}
\author{J. Droniou\thanks{School of Mathematics, Monash University, Melbourne, Australia, {\tt jerome.droniou@monash.edu}}, R. Eymard\thanks{Universit\'e Gustave Eiffel, LAMA, (UMR 8050), UPEM, UPEC, CNRS, F-77454, Marne-la-Vallée (France), {\tt robert.eymard@univ-eiffel.fr}}, T. Gallou\"{e}t\thanks{I2M UMR 7373, Aix-Marseille Universit\'e, CNRS, Ecole Centrale de Marseille,
F-13453 Marseille, France, {\tt thierry.gallouet@univ-amu.fr}}, C. Guichard\thanks{Sorbonne Université, CNRS, Université Paris Cité, Laboratoire
Jacques-Louis Lions (LJLL), F-75005 Paris, France, {\tt cindy.guichard@sorbonne-universite.fr}} and R. Herbin\thanks{ I2M UMR 7373, Aix-Marseille Universit\'e, CNRS, Ecole Centrale de Marseille,
F-13453 Marseille, France, {\tt raphaele.herbin@univ-amu.fr}}
}
\maketitle

\begin{abstract} We consider a general linear parabolic problem with extended time boundary conditions (including initial value problems and periodic ones), and approximate it by the implicit Euler scheme in time and the Gradient Discretisation method in space; the latter is in fact a class of methods that includes conforming and nonconforming finite elements, discontinuous Galerkin methods and several others. 
The main result is an error estimate which holds without supplementary regularity hypothesis on the solution. 
This result states that the approximation error has the same order as the sum of the interpolation error and the conformity error. 
The proof of this result relies on an inf-sup inequality in Hilbert spaces which can be used both in the continuous and the discrete frameworks. 
The error estimate result is illustrated by numerical examples with low regularity of the solution.
\end{abstract}


{\bf AMS class (2010):} {65N30, 35K15, 47A07}

\section{Introduction}
Let us first recall the a priori error estimate which holds for the approximation, by a conforming method, of the homogeneous Dirichlet problem in a bounded nonempty open set $\Omega$ of $\mathbb{R}^d$, $d\in\mathbb{N}^\star$.
Let $\overline{u}\in H^1_0(\Omega)$ and $u_\disc\in U_\disc\subset H^1_0(\Omega)$ (where $U_\disc$ is a finite dimensional vector space), be the respective solutions of 
\[
 \forall v\in H^1_0(\Omega),\ \langle \nabla \overline{u},\nabla v\rangle_{L^2} =  \langle f,v\rangle_{L^2} 
\]
and
\[
 \forall v\in U_\disc,\  \langle \nabla u_\disc,\nabla v\rangle_{L^2} =  \langle f,v\rangle_{L^2},
\]
where $f\in L^2(\Omega)$ is given, and where   $\langle \cdot,\cdot\rangle_{L^2}$ denotes the scalar product in $L^2(\Omega)^d$ or in $L^2(\Omega)$. 
It is well-known that  C\'ea's Lemma \cite{cea1964} yields the following error estimate:
\[
 \inf_{v\in U_\disc} \dist(\overline{u},v) \le \dist(\overline{u},u_\disc) \le (1+ {\rm diam}(\Omega)) \inf_{v\in U_\disc} \dist(\overline{u},v),
\]
where, for any $v\in U_\disc$, $\dist(\overline{u},v)$ measures the distance between the element $\overline{u}\in H^1_0(\Omega)$ and the element $v\in U_\disc$, as defined by
\[
 \dist(\overline{u},v)^2 = \Vert \nabla \overline{u} - \nabla  v\Vert_{L^2}^2+\Vert\overline{u} -  v\Vert_{L^2}^2.
\]
The above error estimate is optimal, since it shows that the approximation error $\dist(\overline{u},u_\disc)$ has the same order as that of the interpolation error $\inf_{v\in U_\disc} \dist(\overline{u},v)$. Such a generic error estimate is then used for determining the order of the method if the solution shows more regularity, leading to an interpolation error controlled by higher order derivatives of the solution.

Turning to approximations of the function $\overline{u}$ which are  nonconforming (\ie~ no longer belonging to the space where the unique solution of the problem lives), we consider the framework of the Gradient Discretisation method (GDM), which encompasses conforming approximations, as well as nonconforming finite elements or discontinuous Galerkin methods  \cite{gdm}. In this framework, the approximation of $\overline{u}$ is defined as $u_\disc\in X_\disc$ (the finite dimensional vector space on $\mathbb{R}$ associated with the degrees of freedom of the approximate solution), such that
\[
 \forall v\in X_\disc,\ \langle \agrad_\disc u_\disc,\agrad_\disc v\rangle_{L^2} =  \langle f,\api_\disc  v\rangle_{L^2}.
\]
In the above equation, for any $v\in X_\disc$, the function  $\api_\disc v\in L^2(\Omega)$ is the function reconstructed from the degrees of freedom, defined on $\Omega$, and $\agrad_\disc v\in L^2(\Omega)^d$ stands for the reconstruction of its approximate gradient. 
Then the following error estimate \cite[Theorem 2.28]{gdm} is a reformulation of G. Strang's second lemma \cite{str-72-var}:
\[
\frac 1 2 \left[\limconf_\disc(\nabla \overline{u}) + \inf_{v\in X_\disc} \dist(\overline{u},v)\right]\le  \dist(\overline{u},u_\disc) \le (1+ p_\disc)  \left[\limconf_\disc(\nabla \overline{u}) + \inf_{v\in X_\disc} \dist(\overline{u},v)\right],
\]
where $\dist(\overline{u},v)$, which measures the distance between the element $\overline{u}\in H^1_0(\Omega)$ and the element $v\in X_\disc$, is such that
\[
 \dist(\overline{u},v)^2 = \Vert \nabla \overline{u} - \agrad_\disc  v\Vert_{L^2}^2+\Vert\overline{u} - \api_\disc  v\Vert_{L^2}^2,
\]
and $\limconf_\disc(\nabla \overline{u})$, which measures the conformity error of the method (it vanishes in the case of conforming methods), is defined by
\[
 \limconf_\disc(\nabla \overline{u}) = \max_{v\in X_\disc\setminus \{0\} } \frac {\langle \nabla \overline{u},\agrad_\disc  v\rangle_{L^2} -  \langle {\rm div}(\nabla \overline{u}),\api_\disc  v\rangle_{L^2}} { \Vert\agrad_\disc  v\Vert_{L^2} }.
\]
The value $p_\disc$ is associated to the discrete Poincar\'e inequality 
\begin{equation}\label{eq:defph}
  \Vert\api_\disc  v\Vert_{L^2} \le p_\disc \Vert \agrad_\disc  v\Vert_{L^2},\hbox{ for all }v\in X_\disc.
\end{equation}
In the case where $p_\disc$ is bounded independently of the accurateness of the approximation (for example, for  mesh-based methods, $p_\disc$ only depends on a regularity factor of the meshes), this error estimate is again optimal: it shows the same order for the approximation error and for the sum of the interpolation and conformity errors.

\medskip

Hence, in the conforming case, the order of the method is only determined by the interpolation properties of $U_\disc$, and in the nonconforming one, by the interpolation and conformity properties of $(X_\disc,\api_\disc,\agrad_\disc)$.

\medskip

In the case of parabolic problems, a large part of the literature only provides error estimates assuming supplementary regularity of the solution. For example, in \cite{droniou2017error}, an error estimate is established for the GDM approximation of the heat equation  under the condition that the exact solution of the problem belongs to the space $W^{1,\infty}(0,T;W^{2,\infty}(\Omega))$. 
Error estimate results for linear parabolic problems in the spirit of C\'ea's Lemma have only recently been published. These results are based on variational formulations of the parabolic problem and on an inf-sup inequality satisfied by the involved bilinear form (see \cite[XVIII.3 Th\'eor\`eme 2]{DL} for first results, and \cite[III Proposition 2.3 p.112]{Sho97} for a more complete formulation); they concern either semi-discrete numerical schemes (continuous in time, discrete in space), see for example \cite{chryhou2002err,tant2016proj}, or fully discrete time-space problems \cite{schwab2009parab,boiv2019app,urban2012error,saito2021var}.
In \cite{arendt2023spacetime}, similar optimal results are obtained for the full time-space approximation of linear parabolic partial differential equation, using Euler schemes or a discontinuous Galerkin scheme in time, together with conforming approximations. 
Let us more precisely describe the result obtained in \cite{arendt2023spacetime}, in the case of the implicit Euler scheme for the heat equation. 
Let $T>0$, $\xi_0\in L^2(\Omega)$ and $f\in L^2(0,T;L^2(\Omega))$ be given and let $\overline{u}\in W := H^1(0,T;H^{-1}(\Omega))\cap L^2(0,T;H^1_0(\Omega))$ (equivalently $W = \{ u \in  L^2(0,T;H^1_0(\Omega)) : \partial_t u \in  L^2(0,T;H^{-1}(\Omega))\}$) be the solution of: $\overline{u}(0) = \xi_0$ and, for a.e. $t\in(0,T)$,
\[
\forall v\in H^1_0(\Omega),\  \langle \partial_t \overline{u}(t),v\rangle_{H^{-1},H^1_0} + \langle \nabla \overline{u}(t),\nabla v\rangle_{L^2} =  \langle f(t),v\rangle_{L^2}.
\]
The existence and uniqueness of $\overline{u}$ is due to J-L. Lions  \cite[Théorème 1.1 p.46]{Lio61}, see also \cite[Théorème 4.29]{castor}. 
Let $N\in\mathbb{N}^\star$ and $U_\disc\subset H^1_0(\Omega)$ be given (as above, $U_\disc$ is assumed to be a finite dimensional vector space). Let $u_\disc := (u^{(m)})_{m=0,\ldots,N}\in W_\disc := U_\disc^{N+1}$ be the solution of: $u^{(0)} = \mathcal{P}^{L^2}_{U_\disc}(\xi_0)$ (orthogonal projection on $U_\disc$ in $L^2(\Omega)$) and, for $m=1,\ldots,N$, 
\[
\forall v\in U_\disc,\ \langle \frac {u^{(m)}- u^{(m-1)}} k,v\rangle_{L^2} + \langle \nabla u^{(m)},\nabla v\rangle_{L^2} =  \langle f^{(m)},v\rangle_{L^2},
\]
with $k=T/N$ and $f^{(m)} = \frac 1 k \int_{(m-1)k}^{mk} f(t){\rm d}t$. Then it is shown in \cite{arendt2023spacetime} that
\[
\inf_{v\in W_\disc} \disT(\overline{u},v)\le  \disT(\overline{u},u_\disc) \le C  \inf_{v\in W_\disc} \disT(\overline{u},v),
\]
where $ \disT(\overline{u},v)$ is a suitable distance between the elements of $W$ and those of $W_\disc$, and $C$ only depends on $T$ and $\Omega$. Note that the common bilinear form, for which inf-sup inequalities cover both the discrete and the continuous case, is not conforming in $W$.

\medskip

The present work establishes an optimal error estimate result for the full time-space approximation of linear parabolic partial differential equation, using the implicit Euler scheme together with the GDM for the approximation of the continuous operators, without assuming a stronger regularity than the natural hypothesis $\overline{u}\in W$. 
Our analysis  also includes conforming methods with mass lumping: the latter technique is widely used, for stability reasons, in the real life implementation of conforming finite element methods for parabolic problems. 
Indeed, the implementation of the mass lumping, often viewed as a numerical integration approximation, is in fact a change of the approximation space which yields a conformity error (see, e.g., the presentation in \cite[Section 8.4]{gdm}), and the resulting implicit Euler method is thus a doubly non conforming method, both in space and in time. 

\medskip

 Let us describe such a doubly non conforming scheme in the case of the discretisation of the heat equation. Given $(X_\disc,\api_\disc,\agrad_\disc)$ for the nonconforming approximation of an elliptic problem by the GDM, the time-space approximation is defined through the knowledge of $u_\disc := (u^{(m)})_{m=0,\ldots,N}\in W_\disc := X_\disc^{N+1}$, solution of:
$\api_\disc u^{(0)} = \mathcal{P}^{L^2}_{\api_\disc(X_\disc)}(\xi_0)$ (orthogonal projection on $\api_\disc(X_\disc)$ in $L^2(\Omega)$) and, for $m=1,\ldots,N$, 
\[
\forall v\in X_\disc,\ \langle \api_\disc \frac {u^{(m)}- u^{(m-1)}} k,\api_\disc v\rangle_{L^2} + \langle \agrad_\disc u^{(m)},\agrad_\disc v\rangle_{L^2} =  \langle f^{(m)},\api_\disc v\rangle_{L^2},
\]
defining $k$ and $f^{(m)}$ as above. 
Then our main result (expressed in Theorem \ref{thm:errest}) states that
\[
\frac 1 2 \left[\limcT_\disc(\bv) + \inf_{v\in W_\disc} \disT(\overline{u},v)\right]\le  \disT(\overline{u},u_\disc) \le C_\disc \left[\limcT_\disc(\bv) + \inf_{v\in W_\disc} \disT(\overline{u},v)\right],
\]
where $\bv\in L^2(0,T;H_{\div}(\Omega))$ is computed from $\overline{u}$ by \eqref{eq:defbv}, $\limcT_\disc(\bv)$ defined by \eqref{abs:defwdiscV} again measures the conformity error of the method (and again vanishes in the case of conforming methods), and $\disT(\overline{u},v)$ measures the distance between the element $\overline{u}\in W$ and the element $v\in W_\disc$ (see \eqref{eq:def.Suv}). 
The real number $C_\disc$ depends continuously on $p_\disc$ (see \eqref{eq:defph}) which remains bounded for any reasonable nonconforming method \cite{gdm}.
This error estimate is established in the case of nonconforming methods for a general parabolic problem with general time conditions which include periodic boundary conditions. 

\medskip

This paper is organized as follows. 
In Section \ref{sec:pbcont}, we establish the continuous framework for parabolic problems with generic Cauchy data (initial or periodic, for example).  
In Section \ref{sec:disc}, we recall the general setting of the gradient discretisation method (GDM) and define the  GDM for the approximation of space-time parabolic problems.
Section \ref{sec:error_estimate} is concerned with Theorem \ref{thm:errest}, which is our main result and which states the error estimate between the space-time GDM approximation and the exact solution under the natural regularity assumptions given by the existence and uniqueness theorem of Section \ref{sec:pbcont}. The proof of this theorem relies on a series of technical lemmas establishing an inf-sup property on a bilinear form involved in the continuous and the discrete formulation.
In Section \ref{sec:inter}, interpolation results are proved on a dense subspace of the solution space, hence leading to convergence results.
Finally, Section \ref{sec:num} provides a numerical confirmation of the error estimate result, on problems with low regularity solutions.
In the examples that are considered here, the conformity error (which in one case includes the effect of mass lumping) is smaller than the interpolation error.

\section{The parabolic problem}\label{sec:pbcont}

Let $\lp$ and $\lpd$ be separable Hilbert spaces;
let $\wunp\subset\lp$ be a dense subspace of $\lp$ and let $\agrad:\wunp\to \lpd$ be a linear operator whose graph $\mathcal{G} = \{ (u,\agrad u), u\in \wunp\}$ is closed in $\lp\times\lpd$.

As a consequence, $\wunp$ endowed with the graph norm $\norm{u}{\wunp,\mathcal{G}}^2 = \norm{u}{\lp}^2 + \norm{\agrad u}{\lpd}^2$ is a Hilbert space continuously embedded in $\lp$.   We assume that the graph norm is equivalent to $\norm{\agrad u}{\lpd}$, which means that there exists a Poincar\'e constant $C_P$ such that
\begin{equation}\label{eq:poincare}
 \norm{u}{\lp}\le C_P \norm{\agrad u}{\lpd}\hbox{ for all }u\in \wunp.
\end{equation}
As a consequence, we use from hereon the norm $\norm{{\cdot}}{\wunp}:= \norm{{\agrad \cdot}}{\lpd}$ on $\wunp$.
Since $\lp\times\lpd$ is separable, $\wunp$ is also separable for the norm $\norm{\cdot}{\wunp}$ (see  \cite[Ch. III]{brezis}).

\begin{remark} In the case of the heat equation, considering homogeneous Dirichlet boundary conditions, we let $\lp = L^2(\Omega)$, $\lpd = L^2(\Omega)^d$ and $\agrad u = \nabla u$. 
 If we consider an initial value problem with homogeneous Neumann boundary conditions, it is possible to consider $\lpd = L^2(\Omega)^d\times L^2(\Omega)$ and $\agrad u = (\nabla u,u)$, using the change of variable $w(t) = \exp(-t)u(t)$. This change of variable is no longer possible in the case of periodic time boundary conditions. Notice that the solution may be periodic in the case of some zero mean value right-hand-side: in this case, it is possible to choose  $\lpd = L^2(\Omega)^d\times \mathbb{R}$ and $\agrad u = (\nabla u,\int_\Omega u(\bm{x}){\rm d}\bm{x})$.
\end{remark}

\medskip

In the following, the notation $\langle \cdot, \cdot \rangle_{Z}$ denotes the inner product in a given Hilbert space $Z$, and $\langle \cdot, \cdot \rangle_{Z',Z}$ denotes the duality action in a given Banach space $Z$ whose dual space is denoted $Z'$. 
Define $\hdiv$ by: 
\be
\hdiv = \{ \bv\in \lpd\,:\, \exists w\in\lp, \forall u\in\wunp, \langle \bv,\agrad u\rangle_{\lpd} + \langle w,u\rangle_{\lp} = 0\}.
\label{abs:defhdiv}\ee
The density of $\wunp$ in $\lp$ implies (and is actually equivalent to) the following property.
\be\label{wunp.prop}
\mbox{For all $w\in\lp$, }(\forall u\in\wunp,  \langle w,u\rangle_{\lp} = 0)\Rightarrow w = 0.
\ee
Therefore, for any $\bv\in\hdiv$, the element $w\in\lp$ whose existence is assumed in \eqref{abs:defhdiv} is unique; this defines a linear operator $\adiv:\hdiv\to\lp$,   so that 
\begin{equation}
 \forall u\in\wunp,\ \forall \bv\in\hdiv,\ \langle \bv,\agrad u\rangle_{\lpd} + \langle \adiv\bv,u\rangle_{\lp} = 0.\label{abs:stokes-formula}
\end{equation}
It easily follows from this that the graph of $\adiv$ is closed in $\lpd\times \lp$, and therefore that, endowed with the graph norm $\norm{\bv}{\hdiv} = \norm{\bv}{{\lpd}} + \norm{\adiv\bv}{\lp}$, $\hdiv$ is a Hilbert space continuously embedded and dense in $\lpd$ (see \cite[Theorem 5.29 p.168]{Kato1995}).

The continuous framework for linear parabolic problems with general time boundary conditions starts by the usual identification of the space $\lp$ with a subspace of ${\wunp}'$ by letting
\[
 \langle y,u\rangle_{\wunp',\wunp} = \langle y,u\rangle_{\lp},\hbox{ for all }y\in \lp, \  u\in \wunp.
\]
This identification yields the Gelfand triple
\[   {\wunp}\stackrel{d}{\hookrightarrow} \lp\hookrightarrow {\wunp}',
\]
where the superscript $d$ recalls that the first embedding is dense.
Let $T>0$, and recall that we may identify  the dual space ${L^2(0,T;\wunp)}'$ with $L^2(0,T;{\wunp}')$ and the space ${L^2(0,T;\lp)}'$ with $L^2(0,T;{\lp})$; hence we have a further Gelfand triple
\[   
    {L^2(0,T;\wunp)}\stackrel{d}{\hookrightarrow} {L^2(0,T;\lp)} \hookrightarrow {L^2(0,T;\wunp)'}.
\] 

The classical space $W$ associated with the Gelfand triple is defined by 
\begin{equation*}
\begin{aligned}
 W = \Big\{u\in {L^2(0,T;\wunp)};\ \exists C\ge 0&\mbox{ such that }
    \langle u, v'\rangle_{{L^2(\lp)}}\le C\Vert v\Vert_{{L^2(\wunp)}}\\
    &\mbox{ for all }v\in C^1_c((0,T);\wunp)\Big\}.
\end{aligned}
\end{equation*}
The ``time derivative'' of $u\in W$ may then be defined as the element of the space $ L^2(0,T;\wunp)'$  identified with the space $L^2(0,T;\wunp')$ such that
\begin{equation*}
  \langle u',v\rangle_{L^2(\wunp)',{L^2(\wunp)}} := -\langle u, v'\rangle_{{L^2(\lp)}} \mbox{ for all }  v\in C^1_c((0,T);\wunp).
\end{equation*}
Note that here as well as in the rest of this paper, for a given space $Z$ we use in the dual products and norms the notation $L^2(Z)$ (resp. $H^p(Z)$ for $p=1,2$) as an abbreviation for $L^2(0,T;Z)$ (resp.  $H^p(0,T;Z)$ for $p=1,2$).
In other words, we can write $W$ as follows, introducing also a Hilbert structure,
\begin{equation*}
\begin{aligned}
&W=H^1(0,T;{\wunp}')\cap L^2(0,T;{\wunp})\\
&\hbox{ with }\Vert v\Vert_W :=\Big(\Vert v\Vert_{L^2(\wunp)}^2+\Vert v'\Vert_{L^2(\wunp)'}^2\Big)^{\nicefrac12}\mbox{ for all } v\in W.
\end{aligned}
\end{equation*}

The space $W$ can be identified with a subspace of $C([0,T];\lp)$ and there exists $C_T>0$ such that
\begin{equation}\label{eq:embWinH}
\sup_{t\in[0,T]}\Vert v(t)\Vert_{\lp}\le C_T \Vert v\Vert_{W},\hbox{ for all }v\in W.
\end{equation}

Recall the following integration by parts formula (\cite[III Corollary 1.1 p.106]{Sho97}). 
\begin{lem}\label{lem:pbcont}  
One has, for all $v,w\in W$,
\begin{multline*}
 \langle v',w\rangle_{L^2(\wunp)',L^2(\wunp)} + \langle w',v\rangle_{L^2(\wunp)',L^2(\wunp)}  \\
 = \langle v(T),w(T)\rangle_{\lp}-\langle v(0),w(0)\rangle_{\lp}.
\end{multline*}
\end{lem}

Let $\Lambda\in L^\infty(0,T;{\mathcal L}(\lpd,\lpd))$ and let $\wn\in\mathcal L(\lpd,\lpd)$ be a symmetric positive definite operator such that there exists $M\ge 1$ and $\alpha>0$ with
\begin{subequations}\label{eq:prop.wnLambda}
\begin{alignat}{3}
\label{eq:wnLambda.M}
\Vert\wn^{-1}\Lambda(t)\Vert\leq{}& M&&\quad\mbox{ for a.e. $t\in (0,T)$},\\
\label{eq:wnLambda.alpha}
\langle \wn^{-1}\Lambda(t) \xi,\xi \rangle_{\lpd}\geq{}& \alpha \|\xi\|_{\lpd}^2&&\quad\mbox{ for a.e. $t\in (0,T)$ and all $\xi\in \lpd$}.
\end{alignat}
\end{subequations}
We also define $\rho>0$ by
\begin{equation}\label{eq:wnLambda.rho}
 \rho = \mathop{\rm ess\ sup}\limits_{t\in(0,T)}\Vert \wn^{-\nicefrac12}\Lambda(t)\wn^{-\nicefrac12}\Vert_{\lpd}.
\end{equation}

\begin{remark}[Role of $\wn$]\label{rk:why.wn}
The role of $\wn$ is to provide a control on the real number $C$ involved in the error estimate \eqref{eq:errestgd}, through a control of the constants $M,\alpha,\rho$  above; $\wn$ should be chosen to make these constants as small as possible -- and, ideally, to compensate for a possible strong anisotropy of $\Lambda$ (that would create large ratios $M/\alpha$ if $\wn$ is absent). 
In the case where $\Lambda$ is a time-independent symmetric coercive operator, a natural choice is $\wn = \Lambda$; then, we can take $\alpha=M=\rho = 1$ in \eqref{eq:prop.wnLambda}, and the constants in the error estimate \eqref{eq:errestgd} are independent of $\Lambda$ (but the norm of the error estimate depends on it, see Definition \eqref{abs:defweightednorm}).
Note that in the case of the heat equation, $\wn= \Id.$ 
\end{remark}

Let $\Phi~:~\lp\to \lp$ be a linear contraction (which means that  $\Vert\Phi v\Vert_{\lp}\le \Vert v\Vert_{\lp}$ for all $v\in \lp$). 
Our aim is to obtain an error estimate for an approximate solution of the following problem. 
Given $g\in L^2(0,T;\wunp')$ and $\xi_0\in \lp$, find \begin{equation*} 
\hbox{find }u\in W\hbox{ s.t.\ } u' - \adiv(\Lambda \agrad u)=g\hbox{ and }u(0)-\Phi u(T)=\xi_0.
\end{equation*}
 Using the identification between $\wunp$ and $\wunp'$ by the Riesz representation theorem, we decompose $g \in L^2(0,T;\wunp')$ as $g = f +\adiv\bF$ with $f\in L^2(0,T;\lp)$, $\bF\in L^2(0,T;\lpd)$. 
 This decomposition is not unique; indeed $f=0$ is always possible, but in several problems of interest, the source term  $g$ belongs to $L^2(0,T;\lp)$.
 Therefore, the problem to be considered reads
\begin{equation}\label{eq:pbcont}
\hbox{find }u\in W\hbox{ s.t.\ } u' - \adiv(\Lambda \agrad u +\bF)=f\hbox{ and }u(0)-\Phi u(T)=\xi_0.
\end{equation}
We introduce the Riesz isomorphism $R:\wunp'\to \wunp$ (which also defines a Riesz isomorphism samely denoted $R:L^2(0,T;\wunp')\to L^2(0,T;\wunp)$) such that
\begin{equation}\label{eq:defR}
 \forall (\xi,v)\in \wunp'\times \wunp,\ \langle \wn\agrad R\xi, \agrad v\rangle_{\lpd}= \langle \xi, v\rangle_{\wunp',\wunp}.
\end{equation}
The problem \eqref{eq:pbcont} is then equivalent to
\begin{equation}\label{eq:pbcontriesz}
\hbox{find }u\in W\hbox{ s.t.\ } -\adiv(\wn\agrad R u' + \Lambda \agrad u  +\bF)=f\hbox{ and }u(0)-\Phi u(T)=\xi_0,
\end{equation}
which contains that $\wn\agrad Ru' + \Lambda \agrad u + \bF\in L^2(0,T;\hdiv)$.

\begin{thm}[\protect{\cite{RDV1}}]\label{thm:5.3}
 For all  $f\in L^2(0,T;\wunp)'$ and $\xi_0\in \lp$, Problem \eqref{eq:pbcont} has a unique solution.
\end{thm}

\section{The space-time discretisation}\label{sec:disc}

\subsection{Space approximation using the Gradient Discretisation Method}\label{sec:gdm}

\begin{definition}[Gradient Discretisation]
\label{def:graddisc}
A gradient discretisation is defined by $\mathcal{D}_\disc = (X_{\disc},\api_\disc,\agrad_\disc)$, where:
\begin{enumerate}
\item The set of discrete unknowns 
$X_{\disc}$ is a finite dimensional vector space on $\R$.
\item The ``function'' reconstruction $\api_\disc~:~X_{\disc}\to \lp$ is a linear mapping that reconstructs, from an element of $X_{\disc}$, an element in $\lp$.
\item The ``gradient'' reconstruction $\agrad_\disc~:~X_{\disc}\to \lpd$ is a linear mapping that reconstructs, from an e\-le\-ment of $X_{\disc}$, an element of $\lpd$.
\item The mapping  $\agrad_\disc$ is such that the  mapping $v \mapsto \norm{\agrad_\disc v}{\lpd}$ defines a norm on $X_\disc$.
\end{enumerate}
\end{definition}

We then define the following weighted norm on $X_\disc$
\be
\norm{v}{\disc} :=\norm{\wn^{\nicefrac12}\agrad_\disc v}{\lpd}
\label{abs:defweightednorm}\ee 
and $p_\disc$ as the norm of $\api_\disc$:
\be
p_\disc =  \max_{v\in X_{\disc}\setminus\{0\}}\frac {\norm{\api_\disc v}{\lp }} {\Vert v \Vert_{\disc}}.
\label{abs:defcoercivity}\ee

\subsection{Description of the Euler scheme}\label{sub:4.1}

We now refer to the framework of Section \ref{sec:pbcont}. In particular, ${\Phi} :\lp\to \lp$ is linear and $\Vert\Phi\Vert\le 1$. Moreover, $f\in L^2(0,T;\lp)$, $\bF\in L^2(0,T;\lpd)$ and $\xi_0\in \lp$ are given.

 Let $N\in\mathbb{N}\setminus\{0\}$ and define the time step (taken to be uniform for simplicity of presentation) $k = \frac {T}{N}$. 
 For all $m=1,\ldots,N$, $\Lambda^{(m)}\in {\mathcal L}(\lpd,\lpd)$ denotes the coercive linear operator given by
\[
\Lambda^{(m)}  = \frac 1 {k}\int_{(m-1)k}^{m k} \Lambda(t) {\rm d}t
\]
and  $f^{(m)}\in \lp$, $\bF^{(m)}\in \lpd$ are defined by
\[
 f^{(m)} = \frac 1 {k}\int_{(m-1)k}^{m k} f(t) {\rm d}t\quad\hbox{ and }\quad\bF^{(m)} = \frac 1 {k}\int_{(m-1)k}^{m k} \bF(t) {\rm d}t.
\]

 The implicit Euler scheme consists in seeking $N+1$ elements of $X_\disc$, denoted by $(w^{(m)})_{m=0,\ldots,N}$, such that
 \begin{subequations}\label{eq:scheme}
\begin{equation}\label{eq:schemeiniex}
\langle \api_\disc w^{(0)} - \Phi \api_\disc  w^{(N)} ,  \api_\disc u\rangle_{\lp}  =  \langle \xi_0 ,  \api_\disc u\rangle_{\lp}\mbox{ for all }u\in X_\disc
\end{equation}
and 
\begin{equation}\label{eq:schemeimp}
\begin{aligned}
\langle \api_\disc\frac{w^{(m)}-w^{(m-1)}}{k}, \api_\disc u\rangle_{ \lp } {}&+ \langle \Lambda^{(m)} \agrad_\disc w^{(m)},\agrad_\disc  u\rangle_{\lpd}\\
 ={}&  \langle f^{(m)},\api_\disc u\rangle_{\lp} -\langle \bF^{(m)},\agrad_\disc u \rangle_{\lpd}\\
&\qquad\mbox{ for all }m=1,\ldots,N \mbox{ and }u\in X_\disc.
\end{aligned}
\end{equation}
\end{subequations}

\begin{remark}\label{rem:wzero}
 The discrete value $w^{(0)}$ is only involved in \eqref{eq:schemeiniex}-\eqref{eq:schemeimp} through $\api_\disc w^{(0)}$. As a consequence, we only prove in the following the uniqueness of $\api_\disc w^{(0)}$. If $\api_\disc:X_\disc\to\lp$ is one-to-one, this shows the uniqueness of $w^{(0)}$; if this operator is not one-to-one, then $w^{(0)}$ is actually not unique.
\end{remark}

Note that, if $\Phi\equiv 0$, the scheme is the usual implicit scheme, and the existence and uniqueness of a solution $(\api_\disc w^{(0)},(w^{(m)})_{m=1,\ldots,N})$ to \eqref{eq:schemeimp} is standard. In the general case, a linear system involving $\api_\disc w^{(0)}$ must be solved, and its invertibility is proved by Theorem \ref{thm:errest}.
 
We now define the space $W_\disc$ of all functions $w~:~[0,T]\to X_\disc$ that are piecewise constant in time in the following way: there exist $N+1$ elements of $X_\disc$, denoted by $(w^{(m)})_{m=0,\ldots,N}$, such that
 \begin{equation}\label{eq:schememtheta}
 \begin{aligned}
   w(0)={}&w^{(0)},\mbox{ and }\\ 
   w(t) ={}&  w^{(m)}\hbox{ for all }t\in ((m-1)k, mk],\hbox{ for all }m=1,\ldots,N. 
 \end{aligned}
 \end{equation}
We observe that the space $W_\disc$ is isomorphic to $X_\disc^{N+1}$, through the mapping $w\mapsto (w(mk))_{m=0,\ldots,N}$.
We define the discrete derivative of $w\in W_\disc$ as follows:
 \begin{equation}\label{eq:defqnevolstrex}
 \begin{aligned}
   \partial w(t) ={}& \frac{w^{(m)}-w^{(m-1)}}{k},\\
   &\qquad\text{ for a.e. } t\in ((m-1)k, mk),\hbox{ for all }m=1,\ldots,N.
  \end{aligned}
 \end{equation}

Define the space $V_\disc$ of all functions $v\in L^2(0,T; X_\disc)$ for which there exist $N$ elements of $X_\disc$, denoted by $(v^{(m)})_{m=1,\ldots,N}$, such that
 \begin{equation}\label{eq:defVdisc}
   v(t) =  v^{(m)}\hbox{ for all }t\in ((m-1)k, mk),\hbox{ for all }m=1,\ldots,N. 
 \end{equation}
 
 \begin{remark}[Difference between $W_\disc$ and $V_\disc$]
 $W_\disc$ and $V_\disc$ are both spaces of piecewise constant functions in time.
 However, functions in $W_\disc$ are defined \emph{pointwise and everywhere}, including at all time steps (and are left-continuous on $[0,T]$), whereas functions in $V_\disc$ are only defined \emph{almost everywhere} on $(0,T)$.
  \end{remark}
Scheme \eqref{eq:schemeiniex}--\eqref{eq:schemeimp} can than be written under the form: 
\begin{equation}\label{eq:schemevar}
\hbox{Find }w_\disc\in W_\disc\hbox{ such that }\forall (v,z)\in V_\disc\times X_\disc, \quad b(w_\disc, (v,z)) = L((v,z)),
\end{equation}
with 
\begin{equation}\label{eq:defb}
\begin{aligned}
b(w_\disc, (v,z)) ={}& \langle \api_\disc \partial w_\disc, \api_\disc v\rangle_{ L^2(\lp) } + \langle \Lambda \agrad_\disc w_\disc,\agrad_\disc  v\rangle_{L^2(\lpd)}\\
&+\langle \api_\disc w_\disc(0) - \Phi \api_\disc  w_\disc(T) ,  \api_\disc z\rangle_{\lp}  
\end{aligned}
\end{equation}
and 
\begin{equation*}
L((v,z)) =  \langle f,\api_\disc  v \rangle_{L^2(\lp)} -\langle \bF,\agrad_\disc  v  \rangle_{L^2(\lpd)}+  \langle \xi_0 ,  \api_\disc z\rangle_{\lp}.
\end{equation*}
\begin{remark}[Role of the test functions]
In \eqref{eq:schemevar}, the function $v\in V_\disc$ tests the evolution equation \eqref{eq:schemeimp} while $z\in X_\disc$ tests (through $\api_\disc z$) the initial/final condition \eqref{eq:schemeiniex}.
\end{remark}

\begin{thm}\label{thm:exuniqdis}
Under the setting of this section, there exists one and only one solution $(\api_\disc w^{(0)},(w^{(m)})_{m=1,\ldots,N})$ to \eqref{eq:schemeiniex}-\eqref{eq:schemeimp} or equivalently to \eqref{eq:schemevar}. For this solution, we denote by $w_\disc$ the element of $W_\disc$ corresponding to $(w^{(m)})_{m=0,\ldots,N})$ for a given choice of $w^{(0)}$.
\end{thm}
\begin{proof}
 Since $(\api_\disc w^{(0)},(w^{(m)})_{m=1,\ldots,N})$ is solution to a square linear system, the error estimate Theorem \ref{thm:errest} proves that, for a null right-hand-side, the solution is null.  Hence the system is invertible.
\end{proof}

\section{Error estimate}
\label{sec:error_estimate}

Define the discrete Riesz operator $R_\disc:X_\disc\to X_\disc$ by: for all $u\in X_\disc$, $R_\disc u$ satisfies
\begin{equation}\label{eq:defRdisc}
\langle \wn\agrad_\disc R_\disc u, \agrad_\disc v\rangle_{\lpd} = \langle \api_\disc u,  \api_\disc v\rangle_{\lp}\quad\mbox{ for all }v\in X_\disc.
\end{equation}
We note that with this definition, the scheme \eqref{eq:schemevar} can be recast as: for all $(v,z)\in V_\disc\times X_\disc$,
\begin{equation}\label{eq:schemevar.recast}
\begin{aligned}
\langle \wn\agrad_\disc R_\disc\partial  w_\disc+\Lambda \agrad_\disc w_\disc+\bF, {}&\agrad_\disc v\rangle_{ L^2(\lpd) } 
+\langle \api_\disc w_\disc(0) - \Phi \api_\disc  w_\disc(T) ,  \api_\disc z\rangle_{\lp}\\
={}&
\langle f,\api_\disc  v \rangle_{L^2(\lp)} +  \langle \xi_0 ,  \api_\disc z\rangle_{\lp}.
\end{aligned}
\end{equation}
Set, for all $u\in W$ and $v\in W_\disc$,
\begin{equation}\label{eq:def.Suv}
\begin{aligned}
 \disT_\disc(u,v) ={}& \Vert \wn^{\nicefrac12}(\agrad Ru' - \agrad_\disc R_\disc\partial v)\Vert_{L^2(\lpd)}\\
 &+\Vert \wn^{-\nicefrac12}\Lambda(\agrad u - \agrad_\disc  v)\Vert_{L^2(\lpd)}+ \max_{t\in[0,T]}\Vert u(t) - \api_\disc  v(t)\Vert_{\lp}.
\end{aligned}
\end{equation}
We also define 
 $\limcT_{\disc}:L^2(0,T;\hdiv) \to [0,+\infty)$  by: for all $\bv\in L^2(0,T;\hdiv)$,
\be
\limcT_{\disc}(\bv) = 
\sup_{v\in V_{\disc}\setminus\{0\}}\frac{\dsp
\left|\langle \bv,\agrad_\disc v\rangle_{L^2(\lpd)} + \langle \adiv\bv,\api_\disc v\rangle_{L^2(\lp)}
\right|}{\Vert  \wn^{\nicefrac12}\agrad_\disc v \Vert_{L^2(\lpd)}}.
\label{abs:defwdiscV}\ee

\begin{thm}\label{thm:errest}
 Let $u$ be the solution to \eqref{eq:pbcont}, let 
\begin{equation}\label{eq:defbv}
 \bv := \wn\agrad Ru' + \Lambda \agrad u  +\bF\in L^2(0,T;\hdiv)
\end{equation}
 and let $w_\disc$ be a solution to \eqref{eq:scheme}. 
 Then there exists $C_\disc\ge 0$, depending only on    $p_\disc$ (defined by \eqref{abs:defcoercivity}) in a non decreasing and continuous way, and on $(\alpha, M, T)$ (see \eqref{eq:prop.wnLambda}), such that:
\begin{multline}\label{eq:errestgd}
\frac 1 2\Big[ \limcT_\disc(\bv)+ \inf_{v\in W_\disc} \disT_\disc(u,v)\Big] \le \\
\disT_\disc(u,w_\disc) 
\le  C_\disc \max(1,\rho)\Big[ \limcT_\disc(\bv) + \inf_{v\in W_\disc} \disT_\disc(u,v)\Big].
\end{multline}
\end{thm}

\begin{remark}[Optimal error estimate]
If $C_\disc$ is bounded independently of $h$, which is the case  if for several discretisation methods for which $p_\disc$ can be shown to be bounded thanks to a regularity assumption on the mesh \cite[Part III]{gdm}, the second inequality in \eqref{eq:errestgd} gives an error estimate for the scheme, while the first inequality shows its optimality. 
This is the result announced in the title and introduction of this work.
\end{remark}

\begin{remark}
By Hypothesis \eqref{eq:wnLambda.alpha} and since $\wn$ is symmetric positive definite, we have
\[
C_\star\left(\norm{\wn \xi}{\lpd}+\norm{\wn^{-\nicefrac12}\Lambda\xi}{\lpd}\right)
\le \norm{\xi}{\lpd}
\le
C^\star\left(\norm{\wn \xi}{\lpd}+\norm{\wn^{-\nicefrac12}\Lambda\xi}{\lpd}\right)
\]
where $C_\star$ and $C^\star$ depend on $\wn$, $\alpha$, $M$. Hence, the estimate \eqref{eq:errestgd} also translates into an estimate
on the term \eqref{eq:def.Suv} without the factors $\wn$ and $\wn^{-\nicefrac12}\Lambda$. The latter estimate, however, has multiplicative constants that may depend more severely on the anisotropy of $\Lambda$, see Remark \ref{rk:why.wn}.
\end{remark}

The proof of Theorem \ref{thm:errest} is given after stating and proving a series of technical lemmas involving operators on Hilbert spaces.

\begin{lem}\label{lem:hypsufbnb}
For $w\in W_\disc$, the following inequalities hold
\begin{multline}\label{eq:majunif}
 \max_{t\in[0,T]} \Vert\api_\disc w(t)\Vert_\lp \\
 \le  \Vert \wn^{\nicefrac12}\agrad_\disc R_\disc\partial w\Vert_{L^2(\lpd)}+\Vert \wn^{\nicefrac12}\agrad_\disc  w\Vert_{L^2(\lpd)}+ \Vert\api_\disc w(0)\Vert_\lp,
\end{multline} 
\begin{equation}\label{eq:hypsufbnb.1}
\langle \wn\agrad_\disc R_\disc\partial w,  \agrad_\disc w\rangle_{L^2(\lpd) }
\ge \frac12 \Vert \api_\disc  w(T)\Vert_\lp^2 - \frac12\Vert\api_\disc  w(0)\Vert_\lp^2,
\end{equation}
and, recalling that $p_\disc$ is defined by \eqref{abs:defcoercivity},
\begin{equation}\label{eq:hypsufbnb.2}
\Vert  \wn^{\nicefrac12}\agrad_\disc R_\disc\partial w\Vert_{L^2(\lpd)}^2 + \left(1+\frac{p_\disc^2}{T}\right) \Vert \wn^{\nicefrac12}\agrad_\disc w\Vert_{L^2(\lpd)}^2
\ge \Vert \api_\disc  w(T)\Vert_\lp^2.
\end{equation}
\end{lem}

\begin{proof}
Let $w\in W_\disc$. Using the relation $(a-b)a=\frac12 a^2+\frac12 (a-b)^2-\frac12 b^2$, the definition \eqref{eq:defRdisc} of $R_\disc$ yields, for $0\le m \le m'\le N$,
\begin{align}
 \int_{mk}^{m'k} \langle {}&\wn\agrad_\disc R_\disc\partial w(t),  \agrad_\disc w(t)\rangle_{\lpd }{\rm d}t =
 \int_{mk}^{m'k} \langle \api_\disc\partial w(t),  \api_\disc w(t)\rangle_{\lp }{\rm d}t \nonumber\\
 ={}& \sum_{p=m}^{m'-1} k \langle \api_\disc \frac {w^{(p+1)} - w^{(p)}} k,\api_\disc  w^{(p+1)}\rangle_\lp \nonumber\\
 ={}& \frac 1 2 \Vert  \api_\disc w^{(m')}\Vert_\lp^2 +\frac 1 2\sum_{p=m}^{m'-1} \Vert  \api_\disc (w^{(p+1)} - w^{(p)})\Vert_\lp^2 -\frac 1 2 \Vert  \api_\disc w^{(m)}\Vert_\lp^2.
\label{eq:ineqestimthetaim}
\end{align}
Using the Cauchy--Schwarz inequality on the left-hand side provides
\begin{align}
 \frac 1 2 {}&\Vert \api_\disc w^{(m')}\Vert_\lp^2  \le{}   \Vert  \wn^{\nicefrac12}\agrad_\disc R_\disc\partial w\Vert_{L^2(\lpd)}\Vert \wn^{\nicefrac12}\agrad_\disc w\Vert_{L^2(\lpd)}+ \frac 1 2 \Vert \api_\disc w^{(m)}\Vert_\lp^2\nonumber\\
 &\le\frac12\Vert  \wn^{\nicefrac12}\agrad_\disc R_\disc\partial w\Vert_{L^2(\lpd)}^2+\frac12\Vert \wn^{\nicefrac12}\agrad_\disc w\Vert_{L^2(\lpd)}^2+ \frac 1 2 \Vert \api_\disc w^{(m)}\Vert_\lp^2,
 \label{eq:est.m.mp}
 \end{align}
where the second line follows from the Young inequality. Setting $m=0$ allows us to take any $m'=0,\ldots,N$. Taking the square root of the above inequality and using $(a^2+b^2+c^2)^{\nicefrac12}\le a+b+c$ then concludes the proof of \eqref{eq:majunif}.

The inequality \eqref{eq:hypsufbnb.1} is obtained letting $m=0$ and $m'=N$ in \eqref{eq:ineqestimthetaim}. To prove \eqref{eq:hypsufbnb.2}, we come back to \eqref{eq:est.m.mp} and set $m'=N$ to get, after multiplication by $2k$, for all $m=0,\ldots,N$,
\[
 k \Vert  \api_\disc w(T)\Vert_\lp^2  \le   k   \Vert  \wn^{\nicefrac12}\agrad_\disc R_\disc\partial w\Vert_{L^2(\lpd)}^2+k\Vert \wn^{\nicefrac12}\agrad_\disc w\Vert_{L^2(\lpd)}^2+ k \Vert \api_\disc w^{(m)}\Vert_\lp^2.
\]
Summing over $m=1,\ldots,N$ yields
\begin{align*}
  T \Vert \api_\disc w(T)\Vert_\lp^2  \le{}& T  \Vert  \wn^{\nicefrac12}\agrad_\disc R_\disc\partial w\Vert_{L^2(\lpd)}^2+T\Vert \wn^{\nicefrac12}\agrad_\disc w\Vert_{L^2(\lpd)}^2 + \Vert \api_\disc w\Vert_{L^2(\lp)}^2\\
  \le{}&T  \Vert  \wn^{\nicefrac12}\agrad_\disc R_\disc\partial w\Vert_{L^2(\lpd)}^2+(T+p_\disc^2)\Vert \wn^{\nicefrac12}\agrad_\disc w\Vert_{L^2(\lpd)}^2,
\end{align*}
which proves \eqref{eq:hypsufbnb.2}.
\end{proof}

\begin{lem}\label{lem:zigoto}
Let ${V}$ be a Hilbert space and let $A~:~{V}\to {V}$ be an $M$-continuous and $\alpha$-coercive operator (with $M\ge 1$ and $\alpha>0$), which means that 
\begin{equation}
\label{eq:Malpha.operator}
\Vert A v\Vert_{{V}} \le M \|v\|_{V}\quad\mbox{and}\quad
\langle A v, v\rangle_{{V}} \ge \alpha \|v\|_{V}^2\quad\forall v\in {V}.
\end{equation}  
Then, for all $v,w\in V$,
 \begin{equation}\label{eq:3.8n}
 \| w + A v\|_{{V}}^2\geq 2\alpha\langle w,v\rangle_{{V}}+ \frac 1 3 \left(\frac \alpha M\right)^3 (\|w \|_{{V}}^2 +\|v\|_{V}^2).
 \end{equation} 
\end{lem}
\begin{proof}
 Consider the symmetric $A_{\sym}:=\frac{A+A^*}{2}$ and anti-symmetric $A_{\asym}:=\frac{A-A^*}{2}$ parts of $A$. We have, for all $v\in {V}$,
  $\langle A_{\sym} v,v\rangle=\langle A v,v\rangle\geq \alpha \|v\|_{V}^2$. It follows that the selfadjoint operator $A_{\sym}$ is positive and invertible, and has a positive invertible square root $A_\sym^{\nicefrac12}$ which satisfies
 \begin{equation}\label{eq:pf1}
 \|A_\sym^{\nicefrac12}v\|_{V}\geq \sqrt{\alpha} \|v\|_{V}\quad\forall v\in {V}
 \end{equation}
 and $\|A_\sym^{\nicefrac12}v\|\leq  \sqrt{M}\|v\| $, so that 
 \begin{equation}\label{eq:pf2}
  \|A_\sym^{-1/2} v\|   \geq \frac{1}{\sqrt{M}} \|v\|\quad\forall v\in {V}.  
 \end{equation} 
 Applying \eqref{eq:pf1} to $A_\sym^{-1/2}(w+A v)=A_\sym^{-1/2}w+A_\sym^{\nicefrac12}v+A_\sym^{-1/2}A_\asym v$ instead of $v$ gives 
 \begin{align*}
  \|w + A v\|^2\  \geq{}&  \alpha (\|A_\sym^{-1/2}w +A_\sym^{\nicefrac12}v+ A_\sym^{-1/2}A_{\asym}v\|_{V}^2)  \\
    ={} &   \alpha( \|A_\sym^{-1/2}w\|_{V}^2 + \|A_\sym^{\nicefrac12}v\|_{V}^2 +\|A_\sym^{-1/2}A_{\asym}v\|_{V}^2 )  \\
     & + 2\alpha \langle w, v\rangle_{V} + 2\alpha \langle v, A_{\asym} v\rangle_{V}
      + 2\alpha \langle A_\sym^{-1/2}w, A_\sym^{-1/2}A_{\asym} v\rangle_{V} 
 \end{align*}
 where the second line follows from developing the square of the norm and using $\langle A_\sym^{-1/2}\cdot,A_\sym^{\nicefrac12}\cdot\rangle_V=\langle \cdot,\cdot\rangle_V$. By anti-symmetry of $A_\asym$ we have  $\langle v, A_{\asym} v\rangle_{V}=0$, which leads to
\begin{align}
	\|w + A v\|^2\ \geq{}&  2\alpha\langle w,  v\rangle_{V}+  \alpha (\|A_\sym^{-1/2}w\|_{V}^2 + \|A_\sym^{\nicefrac12}v\|_{V}^2 +\|A_\sym^{-1/2}A_{\asym}v\|_{V}^2 )  \nonumber\\
	& + 2\alpha  \langle A_\sym^{-1/2}w, A_\sym^{-1/2}A_{\asym} v\rangle_{V}.
	\label{eq:est.w.Av.1}
\end{align}
Now we use the Young inequality combined with the Cauchy-Schwarz inequality to estimate, for all $\gamma>0$,   
\begin{align*}
2 \left|  \langle A_\sym^{-1/2}w, A_\sym^{-1/2}A_{\asym} v\rangle_{V} \right| \leq{}&  2 \|A_\sym^{-1/2}w\|_{V} \|A_\sym^{-1/2}A_{\asym}v\|_{V}\\
  \leq{}& \gamma \|A_\sym^{-1/2}w\|_{V}^2 +\frac{1}{\gamma} \|A_\sym^{-1/2}A_{\asym}v\|_{V}^2.
\end{align*}
Taking $\gamma <1$ and plugging this estimate into \eqref{eq:est.w.Av.1} yields
\begin{multline*}
	\|w+A v\|^2\  \geq   2\alpha\langle w,  v\rangle_{V}+ \alpha (1-\gamma)\|A_\sym^{-1/2}w\|_{V}^2 \\
	+\alpha \left(1-\frac{1}{\gamma}\right)\|A_\sym^{-1/2}A_{\asym}v\|_{V}^2
		+\alpha \|A_\sym^{\nicefrac12}v\|_{V}^2.
\end{multline*}
Applying \eqref{eq:pf1} with $A_\sym^{-1/2}A_\asym v$ instead of $v$ and using $\|A_\asym v\|_V\le M\|v\|_V$ gives
\[
 \|A_\sym^{-1/2}A_{\asym}v\|_{V} \le \frac {M} {\sqrt{\alpha}} \|v\|_{V},
\]
which leads, since $1-\frac{1}{\gamma}<0$, to
\begin{align*}
	\|w+A v\|^2\  \geq   2\alpha\langle w,  v\rangle_{V}+ \alpha (1-\gamma)\|A_\sym^{-1/2}w\|_{V}^2 \\
	+ M^2\left(1-\frac{1}{\gamma}\right)\|v\|_{V}^2
		+\alpha \|A_\sym^{\nicefrac12}v\|_{V}^2.
\end{align*}
Let $\gamma=\frac{1}{1+s}$ where $s>0$ is fixed later.  Then $1-\gamma=\frac{s}{1+s}$ and $1-\frac{1}{\gamma} = -s$ and, using \eqref{eq:pf1} and \eqref{eq:pf2}, it follows that 
\[ 	
  \|w+A v\|^2 \geq 2\alpha \langle w,  v\rangle_{V}+ \frac{s}{1+s}\frac{\alpha}{M}\|w\|_{V}^2 -s M^2 \|v\|_{V}^2+\alpha^2\|v\|_{V}^2.
  \]
Choose $s	=\frac{\alpha^2}{2 M^2}$ to obtain 
\[  	
  \|w+A v\|^2 \geq 2\alpha\langle w,  v\rangle_{V}+  \beta (\|w\|_{V}^2 +\|v\|_{V}^2),  
 \]
where, using $\alpha\le M$ and $1\le M$,
\[
 \beta	=\min \left\{   \frac{\alpha^2}{2}, \frac{\alpha^2}{\alpha^2 +2M^2}\frac{\alpha}{M}   \right\} = \frac{\alpha^2}{\alpha^2 +2M^2}\frac{\alpha}{M}  \ge \frac 1 3 \Big(\frac \alpha M\Big)^3.
\]
 \end{proof}

\begin{lem}\label{lem:peterpaul2}
Let $\lp$ be a Hilbert space, and let $\Phi~:~\lp\to \lp$ be a contraction (which means that $\Vert\Phi\Vert\le 1$). Let $a>0$ and $b\in [0,a]$ be given reals such that $\gamma := a - b\Vert\Phi\Vert^2>0$. The following estimate holds
 \begin{equation}\label{eq:peterpaul2}
  a\| w\|_{\lp}^2 - b \| v\|_{\lp}^2 +  \frac {9 a^2}{\gamma}\| v - \Phi w\|_{\lp}^2 \geq \frac {\gamma} 3 (\| w\|_{\lp}^2+\| v\|_{\lp}^2)\mbox{ for all }v,w\in \lp .
 \end{equation} 
\end{lem}
\begin{proof}
Let $v,w\in \lp$ be given.
 By the Young inequality, we have for any $\mu>0$,
 \[
  -2 \langle v, \Phi w\rangle_{\lp} \ge -\mu \| \Phi w\|_{\lp}^2- \frac 1 \mu \| v\|_{\lp}^2.
 \]
Choosing $\mu>1$, this implies that
\begin{align}
 \Vert v-{\Phi} w\Vert_{\lp}^2 \ge{}&\Vert v\Vert_{\lp}^2\left(1 - \frac 1 \mu\right) - \Vert w\Vert_{\lp}^2\Vert\Phi\Vert^2 (\mu - 1).
\label{eq:est.w.Phiv}
\end{align}
Let  $\beta := b\Vert\Phi\Vert^2 \in [0,a)$ and $ \mu := \frac {\beta + 2 a}{2 \beta + a}\in (1,+\infty)$. Let $\theta>0$ and $\alpha>0$ be such that
\[
 \theta \left(1 - \frac 1 \mu\right) = b + \alpha\quad\mbox{ and }\quad
 \theta (\mu - 1)\Vert\Phi\Vert^2 =  a - \alpha.
\]
Using $\Vert\Phi\Vert^2\le 1$, this system is satisfied for
\begin{equation}\label{eq:lower.alpha}
 \alpha = \frac {(a-\beta)(a+\beta)}{2\beta + a + \Vert\Phi\Vert^2(2a+\beta)}\ge \frac{\gamma}{3}.
\end{equation}
Using the preceding equation and $b\le a$, we get
\[
\theta =  \frac {2a+\beta}{a-\beta}(b +\alpha) \le \frac {9 a^2}{\gamma}.
\]
Invoking \eqref{eq:est.w.Phiv} then gives
\begin{align*}
 a\| w\|_{\lp}^2 - b \| v\|_{\lp}^2 +  {}&\frac {9 a^2}{\gamma}\| v - \Phi w\|_{\lp}^2 \ge a\| w\|_{\lp}^2 - b \| v\|_{\lp}^2 +  \theta\| v - \Phi w\|_{\lp}^2 \\
 \ge{}& (a - \theta (\mu - 1)\Vert\Phi\Vert^2)\| w\|_{\lp}^2+\left(\theta \left(1 - \frac 1 \mu\right) -b\right) \| v\|_{\lp}^2  \\
 ={}& \alpha(\| w\|_{\lp}^2+\| v\|_{\lp}^2).
\end{align*}
Recalling \eqref{eq:lower.alpha} concludes the proof of \eqref{eq:peterpaul2}.
\end{proof}

We now give a sufficient condition for establishing an inf-sup condition on the bilinear form $b$ defined by \eqref{eq:defb}. 
Such a condition is sufficient to obtain an error estimate for conforming schemes (see e.g. \cite{ern-guermond}). 
In the case of the (possibly non-conforming) scheme studied in this paper, it provides an essential step in the error estimate proof (see \eqref{eq:infsup}).

\begin{lem}\label{lem:suffbnb}
Let ${V}$ and $\lp$ be Hilbert spaces. Let $\widehat{Z}$ and $Y$ be the Hilbert spaces defined by $\widehat{Z} = {V}\times {V}\times \lp\times \lp$ and $Y = {V}\times \lp$. Let $A:{V}\to {V}$ be an $M$-continuous and $\alpha$-coercive linear operator in the sense of \eqref{eq:Malpha.operator}.
Let $\Phi:\lp\to \lp$ be a linear operator such that $\Vert \Phi\Vert\le 1$. We define $\widehat{b}~:~\widehat{Z}\times Y\to \mathbb{R}$ by
 \begin{equation}\label{eq:defblemme} 
 \widehat{b}((z_1,z_2,z_3,z_4),(y_1,y_2)) = \langle z_1 + A z_2,y_1\rangle_{{V}} + \langle z_3 - \Phi z_4,y_2\rangle_{\lp},
 \end{equation}
for all $(z_1,z_2,z_3,z_4)\in \widehat{Z}$ and for all $(y_1,y_2)\in Y$.

 Let $\widehat{X}\subset \widehat{Z}$ be a subspace of $\widehat{Z}$. We define the Hilbert spaces $\widehat{X}_1\subset V$, $\widehat{X}_2\subset V$, $\widehat{X}_3\subset \lp$ and $\widehat{X}_4\subset \lp$ by: for $i=1,\ldots,4$,
 \[
  \widehat{X}_i = \overline{\{ x_i: \ x\in \widehat{X}\} },\hbox{ where }x_i\hbox{ is the $i$-th component of }x\in \widehat{Z}.
 \]
 Assume that 
 \begin{equation}\label{eq:hypxunxdeux}
 \widehat{X}_1\subset \widehat{X}_2,
 \end{equation}
 and that there exist $\limconf>0$ and $\delta >0$ such that, for all $x\in\widehat{X}$,
\begin{equation}\label{eq:condlim}
 \langle x_1,x_2\rangle_{{V}} +\frac {\alpha^2} {12\, M^3}(\Vert  x_2\Vert_{V}^2 + \Vert x_1\Vert_{{V}}^2)\ge \mu\Vert x_4\Vert_{\lp}^2 -\nu \Vert x_3\Vert_{\lp}^2,
\end{equation}
for some $\mu\in (0,\limconf]$ and $\nu\in [0,\mu]$ with $\mu - \nu\Vert\Phi\Vert^2 \ge \delta$.
 Then,  there exists $\widehat{\beta}>0$, only depending on $\alpha$, $M$, $\limconf$ and $\delta$ (and not on $\mu$, $\nu$ and $\Vert\Phi\Vert$) such that
 \begin{equation}\label{eq:bnbdisc}
  \sup_{y\in \widehat{X}_2\times \widehat{X}_3, \Vert y\Vert_Y = 1} \widehat{b}(x,y) \ge  \widehat{\beta} \Vert x\Vert_{\widehat{Z}}\quad\forall x\in \widehat{X}.
 \end{equation}
\end{lem}

\begin{proof}
Let $x\in\widehat{X}$.
Let $ P_{3}:\lp\to \widehat{X}_3\subset \lp$ be the orthogonal projection onto $\widehat{X}_3$.
Then, setting
\[
 \mathcal{N}(x) = \sup_{y\in \widehat{X}_2\times \widehat{X}_3, \Vert y\Vert_Y = 1} \widehat{b}(x,y),
\]
we have, using \eqref{eq:hypxunxdeux},
\begin{equation}\label{eq:bnbX0}
	\mathcal{N}(x)^2
	= \Vert x_1 + A x_2\Vert_{{V}}^2 + \Vert P_{3}(x_3 - \Phi x_4)\Vert_{\lp}^2.
	\end{equation}
We then obtain, for $\theta>0$ to be chosen later
\begin{equation}\label{eq:bnbX}
		\mathcal{N}(x)^2
	\ge \frac 1 {\max(1,\theta)} \Vert x_1 + A x_2\Vert_{{V}}^2 + \frac{\theta}{\max(1,\theta)} \Vert P_{3}(x_3 - \Phi x_4)\Vert_{\lp}^2.
	\end{equation}

 We apply Lemma \ref{lem:zigoto} to obtain
\begin{align*}
 \Vert  x_1 +  A x_2\Vert_{{V}}^2  \ge{}& 2\alpha \langle x_1,x_2\rangle_{{V}} + \frac 1 3 \left(\frac \alpha M\right)^3(\Vert  x_2\Vert_{V}^2 + \Vert x_1 \Vert_{{V}}^2 )\\
  \ge{}& 2\alpha(\mu\Vert x_4\Vert_{\lp}^2  - \nu\Vert  x_3\Vert_{\lp}^2)
 + \frac 1 6 \left(\frac \alpha M\right)^3 (\Vert  x_2\Vert_{V}^2 + \Vert x_1 \Vert_{{V}}^2 ),
\end{align*}
where the second line follows from the assumption \eqref{eq:condlim}, after writing $\frac 1 3 \left(\frac \alpha M\right)^3=2\alpha \frac{\alpha^2}{12 M^3}+\frac 1 6 \left(\frac \alpha M\right)^3$.

Together with \eqref{eq:bnbX}, this yields
\begin{align}
 \max(1,\theta)\mathcal{N}(x)^2 \ge 2\alpha(\mu\Vert x_4\Vert_{\lp}^2  - \nu\Vert  x_3\Vert_{\lp}^2)
 + \frac {\alpha^3} {6\, M^3} (\Vert  x_2\Vert_{V}^2 + \Vert x_1 \Vert_{{V}}^2 )\nonumber\\
  + \theta\Vert P_{3}(x_3 - \Phi x_4)\Vert_{\lp}^2.
  \label{eq:est.N}
\end{align}
Noting that $ P_{3}(x_3 - \Phi x_4) = x_3 - P_{3}\Phi x_4$ and that $\Vert P_{3}\circ\Phi \Vert\le \Vert\Phi\Vert \le 1$, we use Lemma \ref{lem:peterpaul2} with $P_3\circ\Phi$ instead of $\Phi$, $a= 2\alpha\mu$ and $b= 2\alpha\nu$, which gives $\gamma \ge 2\alpha(\mu - \nu\Vert\Phi\Vert^2)\ge 2\alpha\delta$. If we set
$\theta = \frac {9 a^2}{\gamma} \le \frac {18\alpha\mu^2}{\mu - \nu\Vert\Phi\Vert^2}\le  \frac {18\alpha\limconf^2}{\delta}$, we get
\[
 2\alpha(\mu\Vert x_4\Vert_{\lp}^2  - \nu\Vert  x_3\Vert_{\lp}^2) + \theta \Vert  P_{3}(x_3 - \Phi x_4)\Vert_{\lp}^2
 \ge  \frac{\gamma} {3} (\Vert x_3\Vert_{\lp}^2+\Vert x_4\Vert_{\lp}^2).
\]
Combined with \eqref{eq:est.N}, this gives
\[
\max\left(1,\frac {18\alpha\limconf^2}{\delta}\right) \mathcal{N}(x)^2
 \ge   \frac {\alpha^3} {6\, M^3}  (\Vert  x_1\Vert_{{V}}^2 + \Vert x_2\Vert_{{V}}^2) + \frac{2\alpha\delta} {3}(\Vert x_3\Vert_{\lp}^2+\Vert x_4\Vert_{\lp}^2),
\]
which leads to \eqref{eq:bnbdisc}.
\end{proof}

Let us now prove the error estimate.

\begin{proof}[Proof of Theorem \ref{thm:errest}.]
Let $v\in V_\disc$ and $z\in X_\disc$ be given. Definition \eqref{abs:defwdiscV} of $\limcT_\disc(\bv)$ give
\[
 \int_0^T \Big(\langle \bv(t),\agrad_\disc  v (t)\rangle_{\lpd} + \langle \adiv\bv(t),\api_\disc v(t)\rangle_{\lp}
\Big){\rm d}t
\le  \limcT_\disc(\bv) \Vert \wn^{\nicefrac12}\agrad_\disc v\Vert_{L^2(\lpd)}.
\]
This yields, using the definition of $\bv$, the relation \eqref{eq:pbcontriesz} (which gives $\adiv \bv=-f$), and \eqref{eq:schemevar.recast},
\begin{multline*}
\int_0^T \Big(\langle  \wn\agrad Ru'(t) +\Lambda(t) \agrad u(t)  - ( \wn\agrad_\disc R_\disc\partial w_\disc(t) + \Lambda(t) \agrad_\disc w_\disc(t)) ,\agrad_\disc v(t)\rangle_{\lpd} \Big){\rm d}t\\
+\langle\xi_0 - (\api_\disc w_\disc(0) - \Phi \api_\disc  w_\disc(T)) ,  \api_\disc z\rangle_{\lp}
\le  \limcT_\disc(\bv) \Vert \wn^{\nicefrac12}\agrad_\disc v\Vert_{L^2(\lpd)}.
\end{multline*}
Using $u(0) - \Phi u(T) = \xi_0$, we get
\begin{multline}
\int_0^T \Big(\langle  \wn\agrad Ru'(t) +\Lambda(t) \agrad  u(t)  - ( \wn\agrad_\disc R_\disc\partial w_\disc(t) + \Lambda(t) \agrad_\disc w_\disc(t)) ,\agrad_\disc v(t)\rangle_{\lpd} 
\Big){\rm d}t
\\
 +\langle u(0) - \Phi u(T) - (\api_\disc w_\disc(0) - \Phi \api_\disc  w_\disc(T)) ,  \api_\disc z\rangle_{\lp}
\\
\le  \limcT_\disc(\bv) \Vert \wn^{\nicefrac12}\agrad_\disc v\Vert_{L^2(\lpd)}.
\label{eq:error.est.1}
\end{multline}
We then take an arbitrary element $\widetilde{v}\in W_\disc$ and notice that, by definition \eqref{eq:def.Suv} of $\disT_\disc(u,\widetilde{v})$ and since $\Phi$ is a contraction, 
\begin{multline*}
\int_0^T \langle \wn[\agrad_\disc R_\disc\partial \widetilde{v}-\agrad R u'](t),\agrad_\disc v(t)\rangle_{\lpd} + \langle \Lambda(t)[\agrad_\disc \widetilde{v}-\agrad u](t) ,\agrad_\disc v(t)\rangle_{\lpd} 
{\rm d}t  \\
+\langle \api_\disc\widetilde{v}(0)-u(0) - \Phi (\api_\disc  \widetilde{v}(T)-u(T)) ,  \api_\disc z\rangle_{\lp} \\
\le \disT_\disc(u,\widetilde{v})\Vert \wn^{\nicefrac12}\agrad_\disc v\Vert_{L^2(\lpd)} +   2\disT_\disc(u,\widetilde{v})\Vert \api_\disc z\Vert_{\lp}.
\end{multline*}
Adding this inequality to \eqref{eq:error.est.1} yields
\begin{multline*}
\int_0^T \langle \wn\agrad_\disc R_\disc\partial (\widetilde{v} - w_\disc)(t) +\Lambda(t) \agrad_\disc (\widetilde{v}(t) - w_\disc(t)) ,\agrad_\disc v(t)\rangle_{\lpd} 
{\rm d}t  \\
\langle \api_\disc(\widetilde{v}(0)  -  w_\disc(0)) - \Phi \api_\disc  (\widetilde{v}(T) -w_\disc(T)) ,  \api_\disc z\rangle_{\lp} \\
\le    
(\disT_\disc(u,\widetilde{v})+\limcT_\disc(\bv))\Vert \wn^{\nicefrac12}\agrad_\disc v\Vert_{L^2(\lpd)} +   2\disT_\disc(u,\widetilde{v})\Vert \api_\disc z\Vert_{\lp}.
\end{multline*}
Using the bilinear form $\widehat{b}$ defined by \eqref{eq:defblemme}, with $V=L^2(0,T;\lpd)$ endowed with the inner product $\langle\wn{\cdot},{\cdot}\rangle_{L^2(\lpd)}$ and $A=\wn^{-1}\Lambda$ (which satisfies \eqref{eq:Malpha.operator} by \eqref{eq:wnLambda.M}--\eqref{eq:wnLambda.alpha}), the preceding inequality implies
\begin{equation}\label{eq:est.key.1}
 \widehat{b}((z_1,z_2,z_3,z_4),(y_1,y_2)) \le \widehat{c}_1 \Vert y_1\Vert_{V} + \widehat{c}_2 \Vert y_2\Vert_{\lp},
\end{equation}
with
\begin{alignat*}{3}
z_1 ={}& \agrad_\disc R_\disc\partial (\widetilde{v} - w_\disc), &\qquad z_2 ={}& \agrad_\disc (\widetilde{v} - w_\disc) , \\
z_3 ={}& \api_\disc(\widetilde{v}(0)  -  w_\disc(0)), &\qquad z_4 ={}& \api_\disc  (\widetilde{v}(T) -w_\disc(T)),\\
y_1 ={}& \agrad_\disc v,&\qquad y_2 ={}& \api_\disc z,\\
\widehat{c}_1 ={}& \disT_\disc(u,\widetilde{v})+\limcT_\disc(\bv),&\qquad
\widehat{c}_2 ={}& 2\disT_\disc(u,\widetilde{v}).
\end{alignat*}
We therefore aim to apply Lemma \ref{lem:suffbnb} with
\[
\widehat{X}=\agrad_\disc(V_\disc)\times \agrad_\disc(V_\disc)\times \api_\disc(X_\disc)\times \api_\disc(X_\disc).
\]
Condition \eqref{eq:hypxunxdeux} is satisfied since $\widehat{X}_1 = \widehat{X}_2$.
Let $\widehat{C}$ be an upper bound of the norm of $\api_\disc$ defined by \eqref{abs:defcoercivity}. Adding \eqref{eq:hypsufbnb.1} to $(1+\frac{\widehat{C}^2}{T})^{-1}\frac{\alpha^2}{12 M^3}\times\eqref{eq:hypsufbnb.2}$ shows that the hypothesis \eqref{eq:condlim} is satisfied with
\[
\mu=\frac12+\left(1+\frac{\widehat{C}^2}{T}\right)^{-1}\frac{\alpha^2}{12 M^3}\quad\mbox{ and }\quad
\nu=\frac12.
\]
We note that $\mu-\nu\|\Phi\|^2\ge \mu-\nu=\left(1+\frac{\widehat{C}^2}{T}\right)^{-1}\frac{\alpha^2}{12 M^3}=:\delta$.
Taking the maximum of \eqref{eq:est.key.1} over all $(y_1,y_2)\in \agrad_\disc(V_\disc)\times \api_\disc(X_\disc)$ with norm in $V\times \lp$ equal to 1,
Lemma \ref{lem:suffbnb} therefore yields $\widehat{\beta}>0$ depending only on $\alpha$, $M$, $\widehat{C}$ and $T$ such that
\begin{multline}\label{eq:infsup}
 \widehat{\beta}\Big( \Vert \wn^{\nicefrac12}\agrad_\disc R_\disc\partial (\widetilde{v}  -  w_\disc)\Vert_{L^2(\lpd)}^2 +\Vert \wn^{\nicefrac12}\agrad_\disc  (\widetilde{v}  -  w_\disc)\Vert_{L^2(\lpd)}^2\\ 
 + \Vert\api_\disc (\widetilde{v}  -  w_\disc)(0)\Vert_\lp^2+ \Vert\api_\disc (\widetilde{v}  -  w_\disc)(T)\Vert_\lp^2\Big)^{\nicefrac12}\\
 \le \left[ \left(\disT_\disc(u,\widetilde{v})+\limcT_\disc(\bv) \right)^2 +  4\disT_\disc(u,\widetilde{v})^2\right]^{\nicefrac12}\\
 \le \disT_\disc(u,\widetilde{v})+\limcT_\disc(\bv)  +  2\disT_\disc(u,\widetilde{v}),
\end{multline}
where we use $(a^2+b^2)^{\nicefrac12}\le a+b$ for positive $a$ and $b$ in the last inequality.
By \eqref{eq:wnLambda.rho} we have 
\[
\Vert \wn^{-\nicefrac12}\Lambda\agrad_\disc  (\widetilde{v}  -  w_\disc)\Vert_{L^2(\lpd)}
\le \rho \Vert\wn^{\nicefrac12}\agrad_\disc  (\widetilde{v}  -  w_\disc)\Vert_{L^2(\lpd)}.
\]
Plugging this into \eqref{eq:infsup} and using \eqref{eq:majunif} in Lemma \ref{lem:hypsufbnb} together with $a+b+c\le \sqrt{3}(a^2+b^2+c^2)^{\nicefrac12}$, we infer
\begin{multline*}
 \widehat{\beta}\Big( \Vert \wn^{\nicefrac12}\agrad_\disc R_\disc\partial (\widetilde{v}  -  w_\disc)\Vert_{L^2(\lpd)} +\Vert \wn^{-\nicefrac12}\Lambda\agrad_\disc  (\widetilde{v}  -  w_\disc)\Vert_{L^2(\lpd)}\\
  + \max_{t\in[0,T]} \Vert\api_\disc (\widetilde{v}  -  w_\disc)(t)\Vert_\lp\Big)\\
 \le \sqrt{3}\max(1,\rho)\left( 3\disT_\disc(u,\widetilde{v})+\limcT_\disc(\bv)\right).
\end{multline*}
Using the triangle inequality in the definition \eqref{eq:def.Suv} of $\disT_\disc$, we infer
\[
 \widehat{\beta} \disT_\disc(u, w_\disc)
 \le \sqrt{3}\max(1,\rho)\left( 3\disT_\disc(u,\widetilde{v})+\limcT_\disc(\bv)\right)+ \widehat{\beta}\disT_\disc(u, \widetilde{v}).
\]
Since $\widetilde{v}$ is arbitrary in $W_\disc$, this concludes the proof of the second inequality in \eqref{eq:errestgd}.

\medskip

Let us now turn to the first inequality in \eqref{eq:errestgd}. We first note that
\[
 \inf_{v\in W_\disc} \disT_\disc(u,v) \le \disT_\disc(u,w_\disc).
\]
To bound $\limcT_\disc(\bv)$ we recall that $\bv=\wn\agrad Ru' + \Lambda \agrad u  +\bF$ satisfies $-\adiv \bv=f$ (see \eqref{eq:pbcontriesz}), and use the scheme \eqref{eq:schemevar.recast} (with $z=0$) to write, for any $v\in V_\disc\backslash\{0\}$,
\begin{align*}
\langle \bv,\agrad_\disc v\rangle_\lpd+{}&\langle\adiv\bv,\api_\disc v\rangle_L\\
={}& \int_0^T \Big(\langle  \wn\agrad Ru'(t) +\Lambda(t) \agrad  u(t)  \\
 &\qquad\qquad- ( \wn\agrad_\disc R_\disc\partial w_\disc(t) + \Lambda(t) \agrad_\disc w_\disc(t)) ,\agrad_\disc v(t)\rangle_{\lpd} 
\Big){\rm d}t \\
\le{}& \left(\Vert \wn^{\nicefrac12}(\agrad Ru' - \agrad_\disc R_\disc\partial w_\disc)\Vert_{L^2(\lpd)} + \Vert \wn^{-\nicefrac12}\Lambda(\agrad  u - \agrad_\disc w_\disc)\Vert_{L^2(\lpd)}\right)\\
&\times\Vert \wn^{\nicefrac12}\agrad_\disc v\Vert_{L^2(\lpd)}.
\end{align*}
Dividing by $\Vert \wn^{\nicefrac12}\agrad_\disc v\Vert_{L^2(\lpd)}$ and taking the supremum over $v\in V_\disc\backslash\{0\}$ shows that $\limcT_\disc(\bv)\le \disT_\disc(u,w_\disc)$, which concludes the proof.
\end{proof}

\section{Interpolation results}\label{sec:inter}

In this section, we consider a sequence $(\mathcal D_{h_n})_{n\in\N}$ of gradient discretisations which is 
\begin{enumerate}
 \item \emph{Consistent}, in the sense that 
\be
\forall \varphi\in \wunp\,,\ \lim_{n\to\infty} \consist_{\disc_n}(\varphi)=0,
\label{abs:strongconsistn}
\ee
where
\be\label{def:sigmah}
\begin{aligned}
\consist_{\disc_n}(\varphi) ={}& \inf_{v\in X_{\disc_n}}\delta_{\disc_n}(\varphi,v),\\
\mbox{ with }&
\delta_{\disc_n}(\varphi,v) = \Bigl(\norm{\api_{\disc_n} v - \varphi}{\lp }^2 + \norm{\agrad_{\disc_n} v-\agrad\varphi}{\lpd}^2\Bigr)^{\nicefrac12}.
\end{aligned}
\ee
\item \emph{Limit-conforming}, in the sense that
\be
\forall \bvarphi\in \hdiv\,,\ \lim_{n\to\infty} \limconf_{\disc_n}(\bvarphi) = 0,
\label{abs:limconfn}\ee
where
\be
\forall \bvarphi\in \hdiv \,,\;
\limconf_{\disc_n}(\bvarphi) = \sup_{v\in X_{\disc_n}\setminus\{0\}}\frac{\dsp
\left|\langle \bvarphi,\agrad_{\disc_n} v\rangle_{\lpd} + \langle \adiv\bvarphi,\api_{\disc_n} v\rangle_{\lp}
\right|}{\Vert  v \Vert_{\disc_n}}.
\label{abs:defwdiscn}\ee
\end{enumerate}
Applying \cite[Lemma 3.10]{DEGH18} or \cite[Lemma 2.6]{gdm} for example, we can state that there exists $\widehat{C}>0$ such that, for all $n\in\N$,
\be
\max_{v\in X_{\disc_n}\setminus\{0\}}\frac {\norm{\api_{\disc_n} v}{\lp }} {\Vert \agrad_\disc v\Vert_\lpd}\le \widehat{C}.
\label{abs:defcoercivityn}\ee 
In the following, we denote by $C_i$, for $i\in\mathbb{N}$, various constants which only depend on $\widehat{C}$, $T$, $C_T$ (see \eqref{eq:embWinH}), $\Lambda$ and $\wn$.

Let $N_n$ a sequence of positive integers diverging to infinity, and let $k_n = T/N_n$.
This section is devoted to the proof of the following theorem, which enables us to apply Theorem \ref{thm:errest} for proving the convergence of the scheme under the hypotheses of this section.
\begin{thm}\label{lem:WT} Under the hypotheses of this section, the following holds.

For any $\bvarphi\in L^2(0,T;\hdiv)$,
\begin{equation}\label{eq:WT}
 \lim_{n\to\infty} \limcT_{\disc_n}(\bvarphi) = 0.
\end{equation}

Moreover, recalling the definition \eqref{eq:def.Suv} of $\disT_\disc$, we have, for all $w\in W$,
\begin{equation}\label{eq:ST}
 \lim_{n\to\infty} \inf_{v\in W_{\disc_n}} \disT_{\disc_n}(w,v) = 0.
\end{equation}

As a consequence, letting $u$ be the solution to \eqref{eq:pbcont}, and $w_{\disc_n}$ be the solution to \eqref{eq:scheme} for $\disc = \disc_n$, then
\begin{equation}\label{eq:cvscheme}
 \lim_{n\to\infty} \disT_{\disc_n}(u,w_{\disc_n}) = 0.
\end{equation}

\end{thm}
\begin{proof}
For a.e.\ $t\in (0,T)$ and all $v\in V_{h_n}$ we have
\[
\left|\langle \bvarphi(t),\agrad_{h_n} v(t)\rangle_{\lpd}+\langle\adiv\bvarphi(t),\api_{h_n} v(t)\rangle_\lp\right|
\le \limconf_{h_n}(\bvarphi(t))\|v(t)\|_{h_n}.
\]
Recalling that $\|v(t)\|_{h_n}=\|\wn^{\nicefrac12}\agrad_{h_n}v(t)\|_{\lpd}$, integrating over $t\in (0,T)$ and using the Cauchy--Schwarz inequality yields
\begin{multline*}
\left|\langle \bvarphi,\agrad_{h_n} v\rangle_{L^2(\lpd)}+\langle\adiv\bvarphi,\api_{h_n} v\rangle_{L^2(\lp)}\right|\\
\le \left(\int_0^T \limconf_{h_n}(\bvarphi(t))^2\,dt\right)^{\nicefrac12}\|\wn^{\nicefrac12}\agrad_{h_n}v\|_{L^2(\lpd)}
\end{multline*}
and thus
\[
\limcT_{\disc_n}(\bvarphi)\le \left(\int_0^T \limconf_{h_n}(\bvarphi(t))^2\,dt\right)^{\nicefrac12}.
\]
By limit-conformity we know that, for a.e.\ $t\in (0,T)$, $\limconf_{h_n}(\bvarphi(t))\to 0$ as $n\to\infty$.
Since we also have $\limconf_{\disc_n}(\bvarphi(t)) \le \Vert \bvarphi(t)\Vert_{\hdiv} (\|\wn^{-\nicefrac12}\|+ \widehat{C})$,
we can apply the dominated convergence theorem to obtain \eqref{eq:WT}.

 \medskip
 
Let us now turn to the proof of \eqref{eq:ST}.
Let  $\overline{w}\in H^2(0,T;\wunp)$. We prove in Lemma \ref{lem:errconswreg} that
\[
 \lim_{n\to\infty} \inf_{v\in W_{\disc_n}} \disT_{\disc_n}(\overline{w},v) = 0.
 \]
 The conclusion follows by density of $H^2(0,T;\wunp)$ in $W$, and the property
 \[
 \disT_{\disc_n}(w,v)\le \disT_{\disc_n}(\overline{w},v)+
 \left(\|w-\overline{w}\|_\lp^2 + \|\agrad w-\agrad \overline{w}\|_{\lpd}^2\right)^{\frac12},
 \]
 valid for any $w,\overline{w}\in W$. 
 
 \medskip
 
 Finally, \eqref{eq:cvscheme} is a consequence of \eqref{eq:WT}, \eqref{eq:ST} and Theorem \ref{thm:errest}.
\end{proof}

The next lemmas are steps for the proof of the final lemma of this section, Lemma \ref{lem:errconswreg}.
 
In the following, for legibility reasons, we sometimes drop the index $n$ in $\disc_n$.
Recalling the definition \eqref{def:sigmah} of $\consist_\disc$, we set $\conssT_{\disc}: L^2(0,T;\wunp)\to [0,+\infty)$  as
 \begin{equation*}
\conssT_{\disc}(v) :=  \Vert \consist_\disc(v(\cdot))\Vert_{L^2((0,T))}\mbox{ for all }v\in L^2(0,T;\wunp).
 \end{equation*}
Note that $\conssT_\disc(v)=\Vert \inf_{w\in X_\disc} \delta_{\disc}(v(\cdot),w)\Vert_{L^2((0,T))}$ is not equivalent to $\inf_{w\in W_\disc} \disT_{\disc}(v,w)$ (in particular, it does not include a term equivalent to $\sup_{[t\in[0,T]}\Vert v(t) - w(t)\Vert_\lp$).
 
 We have the following lemma
\begin{lem}\label{lem:consistT}
For any $\varphi\in L^2(0,T;\wunp)$,
\begin{equation}\label{eq:consistT}
 \lim_{n\to\infty} \conssT_{\disc_n}(\varphi) = 0.
\end{equation}
\end{lem}
\begin{proof}
By consistency \eqref{abs:strongconsistn}, for a.e.\ $t\in (0,T)$ we have $\consist_{\disc_n}(\varphi(t))\to 0$ as $n\to\infty$.
Since $0\in X_\disc$ and $\Vert \varphi(t)\Vert_L\le C_P \Vert \agrad \varphi(t)\Vert_{\lpd}$, we also have
$\conssT_{\disc_n}(\varphi(t)) \le (1+C_P)\Vert \varphi(t)\Vert_{\wunp}$. The dominated convergence theorem then concludes the proof of \eqref{eq:consistT}.
\end{proof}

The interpolator $\mathcal{I}_\disc : \wunp\to X_\disc$ is the linear map defined by
\[
\forall u\in\wunp,\ \mathcal{I}_\disc u = \mbox{ argmin}\{\Vert \api_\disc v -  u\Vert_{\lp}^2+\Vert \agrad_\disc v - \agrad u\Vert_{\lpd}^2\,:\,v\in X_\disc\}.
\]
Since $\mathcal I_\disc u$ is the solution of an unconstrained quadratic minimisation problem, we have
\begin{multline*}
\forall u\in\wunp,\ \forall v\in X_\disc,\\
 \langle \api_\disc \mathcal{I}_\disc u,\api_\disc v\rangle_\lp +\langle \agrad_\disc \mathcal{I}_\disc u,\agrad_\disc v\rangle_\lpd = \langle  u,\api_\disc v\rangle_\lp+\langle \agrad u,\agrad_\disc v\rangle_\lpd.
\end{multline*}
Selecting $v=\mathcal I_\disc u$ and using \eqref{eq:poincare} and \eqref{abs:defcoercivityn}, we deduce the bound
\begin{equation}\label{eq:proj.ineq}
\|\agrad_\disc\mathcal I_\disc u\|_\lpd\le (C_P\widehat{C}+1)\|\agrad u\|_\lpd.
\end{equation}
We also define an interpolator for space-time functions: if $w\in C([0,T];\wunp)$, the element $\mathcal I_{\disc,k} w\in W_\disc$ is defined by the relations  \eqref{eq:schememtheta} using the family $(w^{m})_{m=0,\ldots,N}=(\mathcal I_\disc w(mk))_{m=0,\ldots,N}$.
 We then have the following lemma.
 
 \begin{lem}\label{lem:interpinVex} 
 For all $w\in H^1(0,T;\wunp)$ we have
 \[
 \conssT_{\disc}(w) \le \Vert \delta_\disc(w(\cdot),\mathcal I_{\disc,k} w (\cdot))\Vert_{L^2((0,T))} \le \conssT_{\disc}(w) + \ctel{cte:un} k  \Vert w'\Vert_{L^2(\wunp)}.
 \]
 \end{lem}
 \begin{proof}
Recalling the definition \eqref{def:sigmah} of $\delta_\disc$ and using triangle inequalities, we have
\begin{multline}
 \Vert \delta_\disc(w(\cdot),\mathcal I_{\disc,k}w (\cdot))\Vert_{L^2((0,T))} \le \Vert \delta_\disc(w(\cdot),\mathcal{I}_\disc w(\cdot)))\Vert_{L^2((0,T))} \\ +
 \Bigl(\int_0^T \big(\norm{\api_\disc \mathcal I_{\disc,k}w (t) - \api_\disc \mathcal{I}_\disc w(t)}{\lp }^2 + \norm{\agrad_\disc \mathcal I_{\disc,k}w (t)-\agrad_\disc\mathcal{I}_\disc w(t)}{\lpd}^2\big){\rm d}t\Bigr)^{\nicefrac12}\\
 \le \conssT_{\disc}(w)+(\widehat{C}+1)^{\nicefrac12}  \Bigl(\int_0^T \norm{\agrad_\disc \mathcal I_{\disc,k}w (t)-\agrad_\disc\mathcal{I}_\disc w(t)}{\lpd}^2{\rm d}t\Bigr)^{\nicefrac12}.
\label{eq:equiv.1}\end{multline}
For all $m=0,\ldots,N-1$ and for a.e. $t\in((m-1)k,m k)$, it holds
\begin{multline}\label{eq:equiv.loin}
 \norm{\agrad_\disc  \mathcal I_{\disc,k}w (t)-\agrad_\disc\mathcal{I}_\disc w(t)}{\lpd} = \norm{\agrad_\disc\mathcal{I}_\disc  w(mk)-\agrad_\disc\mathcal{I}_\disc w(t)}{\lpd}\\
 =  \norm{\int_{t}^{mk}\agrad_\disc\mathcal{I}_\disc  w'(s){\rm d}s}{\lpd} \le  \int_{t}^{mk}\norm{\agrad_\disc\mathcal{I}_\disc  w'(s)}{\lpd}{\rm d}s.
\end{multline}
This yields, owing to the Cauchy-Schwarz inequality,
\[
 \norm{\agrad_\disc \mathcal I_{\disc,k}w (t)-\agrad_\disc\mathcal{I}_\disc w(t)}{\lpd}^2 \le k \int_{mk}^{(m+1)k} \norm{\agrad_\disc\mathcal{I}_\disc  w'(s)}{\lpd}^2{\rm d}s,
\]
and therefore
\[
  \int_{mk}^{(m+1)k}\norm{\agrad_\disc \mathcal I_{\disc,k}w (t)-\agrad_\disc\mathcal{I}_\disc w(t)}{\lpd}^2{\rm d}t  \le k^2 \int_{mk}^{(m+1)k} \norm{\agrad_\disc\mathcal{I}_\disc  w'(s)}{\lpd}^2{\rm d}s.
\]
Invoking the projection inequality \eqref{eq:proj.ineq} we can write $\norm{\agrad_\disc\mathcal{I}_\disc  w'(s)}{\lpd}\le (C_P\widehat{C}+1)\norm{\agrad  w'(s)}{\lpd}$. Plugging this into the relation \eqref{eq:equiv.1} concludes the proof of the lemma.
 \end{proof}

 \begin{lem}\label{lem:estir}  
For all $u\in \wunp$, recalling the definitions \eqref{eq:defR} and \eqref{eq:defRdisc} of the continuous and discrete Riesz operators, we have
 \begin{equation*}
  \Vert \agrad_\disc R_\disc \mathcal{I}_\disc u - \agrad R u\Vert_\lpd \le \ctel{cte:deux} (\limconf_\disc(\agrad Ru) + \consist_\disc(R u) + \consist_\disc(u) ).
 \end{equation*}
 \end{lem}
 \begin{proof}
 Let $v_1\in X_\disc$ be such that
 \[
  \forall z\in X_\disc,\ \langle \wn \agrad_\disc v_1,\agrad_\disc z\rangle_\lpd = \langle u,\api_\disc z\rangle_\lp.
 \]
 By definition \eqref{eq:defR} of $Ru$, we note that $v_1$ is the solution of the gradient scheme for the linear problem satisfied by $Ru$; hence, we have the following error estimate \cite[Theorem 5.2]{DEGH18}:
 \begin{equation}\label{eq:approx.v1.Ru}
  \Vert \agrad_\disc v_1 - \agrad Ru\Vert_\lpd \le \ctel{cte:errestspace}(\limconf_\disc(\agrad Ru) + \consist_\disc(Ru)).
 \end{equation}
 Recall that $R_\disc\mathcal I_\disc u$ satisfies, by definition of $R_\disc$,
 \[
  \forall z\in X_\disc,\ \langle \wn \agrad_\disc (R_\disc\mathcal I_\disc u),\agrad_\disc z\rangle_\lpd = \langle \api_\disc \mathcal{I}_\disc u,\api_\disc z\rangle_\lp.
 \]
 Subtracting the equations satisfied by $v_1$ and $R_\disc\mathcal I_\disc u$, taking $z=v_1-R_\disc\mathcal I_\disc u$ and using the Cauchy--Schwarz inequality together with \eqref{abs:defcoercivityn}, we obtain
 \[
  \Vert \agrad_\disc(v_1 -  R_\disc\mathcal I_\disc u )\Vert_\lpd \le \ctel{cst:new.1} \Vert u - \api_\disc \mathcal{I}_\disc u\Vert_\lp \le \cter{cst:new.1} \consist_\disc(u).
 \]
 Combined with \eqref{eq:approx.v1.Ru}, this concludes the proof.
 \end{proof}

 \begin{lem}\label{lem:cvschemeex}  
For all $w\in H^2(0,T;\wunp)$, it holds
 \begin{multline*}
  \Vert \agrad R w' - \agrad_\disc  R_\disc\partial \mathcal I_{\disc,k}w \Vert_{L^2(\lpd)}  \\
  \le  k \Vert \agrad R w''\Vert_{L^2(\lpd)} + \cter{cte:deux} \left(\limcT_\disc(\agrad R w') + \conssT_{\disc}(R w') + \conssT_{\disc}(w') \right).
 \end{multline*}
 \end{lem}
 
 \begin{proof}
 Let $\overline{w'}_\disc$ be the function defined on $(0,T)$ by: for all $m=0,\ldots,N$ and $t\in((m-1)k,mk)$,
 \[
  \overline{w'}_\disc(t) = \frac 1 k\int_{(m-1)k}^{mk} w'(s){\rm d}s = \frac {w(mk) - w((m-1)k)} k.
 \]
 We have
 \begin{multline}
  \Vert \agrad R w' - \agrad_\disc R_\disc\partial \mathcal I_{\disc,k}w \Vert_{L^2(\lpd)} \\
  \le  \Vert \agrad R w' - \agrad R\overline{w'}_\disc\Vert_{L^2(\lpd)} + \Vert \agrad R\overline{w'}_\disc - \agrad_\disc  R_\disc\partial \mathcal I_{\disc,k}w \Vert_{L^2(\lpd)}.
\label{eq:est.H2.1}\end{multline}
We have
\[
  \Vert \agrad R w' - \agrad R\overline{w'}_\disc\Vert_{L^2(\lpd)}^2 = \sum_{m=0}^{N-1} \int_{mk}^{(m+1)k} \norm{ \agrad R w'(t) - \frac 1 k\int_{(m-1)k}^{mk}  \agrad R w'(s){\rm d}s}{\lpd}^2{\rm d}t.
\]
We have, using the Jensen inequality,
\begin{multline*}
 \int_{mk}^{(m+1)k} \norm{ \frac 1 k\int_{(m-1)k}^{mk} (\agrad R w'(t) -  \agrad R w'(s)){\rm d}s}{\lpd}^2{\rm d}t \\
 \le \int_{mk}^{(m+1)k}\int_{(m-1)k}^{mk} \frac 1 k\norm{ \agrad R w'(t) -   \agrad R w'(s)}{\lpd}^2{\rm d}s{\rm d}t,
\end{multline*}
and
\[
 \norm{ \agrad R w'(t) -   \agrad R w'(s)}{\lpd}^2 \le k\int_{mk}^{(m+1)k}\norm{ \agrad R w''(\tau)}{\lpd}^2{\rm d}\tau.
\]
This yields
\begin{equation}\label{eq:est.H2.2}
  \Vert \agrad R w' - \agrad R\overline{w'}_\disc\Vert_{L^2(\lpd)} \le k \Vert \agrad R w''\Vert_{L^2(\lpd)}.
\end{equation}
On the other hand, for a.e. $t\in((m-1)k,mk)$ and writing 
\[
\partial \mathcal I_{\disc,k}w(t)=\frac 1 k (\mathcal{I}_\disc w(mk) - \mathcal{I}_\disc w((m-1)k)) = \frac 1 k \int_{(m-1)k}^{mk}\mathcal{I}_\disc w'(s){\rm d}s,
\]
we have
\[
 \agrad R\overline{w'}_\disc(t) - \agrad_\disc R_\disc \partial \mathcal I_{\disc,k}w (t)  = \frac 1 k\int_{(m-1)k}^{mk} \Big(\agrad R w'(s) - \agrad_\disc R_\disc \mathcal{I}_\disc  w'(s)\Big){\rm d}s.
\]
This yields
\[
 \Vert \agrad \overline{w'}_\disc - \agrad_\disc  R_\disc\partial \mathcal I_{\disc,k}w \Vert_{L^2(\lpd)} \le \Vert \agrad R w' - \agrad_\disc  R_\disc\mathcal{I}_\disc  w'\Vert_{L^2(\lpd)}.
\]
Applying Lemma \ref{lem:estir} to $u=w'(t)$, squaring and integrating over $t\in (0,T)$, we infer
\[
 \Vert \agrad \overline{w'}_\disc - \agrad_\disc  R_\disc\partial \mathcal I_{\disc,k}w \Vert_{L^2(\lpd)} \le  \cter{cte:deux} (\limcT_\disc(\agrad R w') + \conssT_{\disc}(R w') + \conssT_{\disc}(w') ).
\]
The proof is concluded by combining this estimate, \eqref{eq:est.H2.1} and \eqref{eq:est.H2.2}.
  \end{proof}

 \begin{lem}\label{lem:interpinVpex}  For all $w\in H^2(0,T;\wunp)$, we have
 \begin{multline*}
 \sup_{t\in[0,T]}\Vert \api_\disc \mathcal I_{\disc,k}w (t) - w(t)\Vert_\lp \\
 \le  \ctel{cte:supt} \Big(k \left(\Vert w'\Vert_{L^2(\wunp)}+ \Vert w''\Vert_{L^2(\wunp)}\right)+ \conssT_{\disc}(w) + \conssT_{\disc}(w')\Big).
 \end{multline*}
 \end{lem}
 \begin{proof}
 Let us first establish a preliminary inequality. For $s,t\in[0,T]$,
\[
 w'(t) = w'(s) + \int_s^t w''(\tau){\rm d}\tau,
\]
which leads to 
\[
\Vert w'(t)\Vert_{\wunp} \le \Vert w'(s)\Vert_{\wunp} + \int_0^T  \Vert w''(\tau)\Vert_{\wunp}{\rm d}\tau.
\]
Integrating with respect to $s$ and using the Cauchy-Schwarz inequality, we obtain
\begin{equation}\label{eq:poincmoy}
 \sup_{t\in[0,T]} \Vert w'(t)\Vert_{\wunp}\le \frac 1 {\sqrt{T}} \Vert w'\Vert_{L^2(\wunp)}+ \sqrt{T}\Vert w''\Vert_{L^2(\wunp)}.
\end{equation}

 For all $t\in[0,T]$, we have
 \begin{multline}
  \Vert \api_\disc \mathcal I_{\disc,k}w (t) - w(t)\Vert_\lp \le \Vert \api_\disc \mathcal I_{\disc,k}w (t) -\api_\disc  \mathcal{I}_\disc w(t)\Vert_\lp + \Vert \api_\disc \mathcal{I}_\disc w(t) - w(t)\Vert_\lp \\
  \le \widehat{C} \Vert \agrad_\disc (\mathcal I_{\disc,k}w (t) - \mathcal{I}_\disc w(t))\Vert_{\lpd} + \Vert \api_\disc \mathcal{I}_\disc w(t) - w(t)\Vert_\lp.
\label{eq:est.PhIh}
 \end{multline}

The first term in the right-hand side can be bounded using \eqref{eq:equiv.loin}, \eqref{eq:proj.ineq} (with $u=w'(t)$) and  \eqref{eq:poincmoy} to write
\begin{multline}
 \Vert \agrad_\disc (\mathcal I_{\disc,k}w (t) - \mathcal{I}_\disc w(t))\Vert_{\lpd} \\
 \le (C_P\widehat{C}+1)k\left(\frac 1 {\sqrt{T}} \Vert w'\Vert_{L^2(\wunp)}+ \sqrt{T}\Vert w''\Vert_{L^2(\wunp)}\right).
\label{eq:est.PhIh.2}\end{multline}
To estimate the second term in the right-hand side of \eqref{eq:est.PhIh}, we write, for any $s\in (0,T)$,
\[
 \api_\disc  \mathcal{I}_\disc w(t) - w(t) =  \api_\disc  \mathcal{I}_\disc w(s) - w(s) + \int_s^t  (\api_\disc  \mathcal{I}_\disc w'(\tau) - w'(\tau)){\rm d}\tau.
\]
Integrating with respect to $s$ and using the Cauchy-Schwarz inequality, this yields
\[
  \Vert  \api_\disc  \mathcal{I}_\disc w(t) - w(t)\Vert_{\lp} \le \ctel{cte:depT}\left( \conssT_{\disc}(w)+\conssT_{\disc}(w')\right).
\]
Plugging this estimate together with \eqref{eq:est.PhIh.2} in  \eqref{eq:est.PhIh} concludes the proof.
\end{proof}

\begin{lem}\label{lem:errconswreg}
For all $w\in H^2(0,T;\wunp)$, it holds
\begin{multline}\label{eq:errestreg}
 \disT_{\disc}(w,\mathcal I_{\disc,k}w) \le  \ctel{cte:errestreg}\Big( k  \left(\Vert w'\Vert_{L^2(\wunp)}+ \Vert w''\Vert_{L^2(\wunp)}\right)\\
  +  \limcT_\disc(\agrad R w') + \conssT_{\disc}(R w') + \conssT_{\disc}(w') + \conssT_{\disc}(w) \Big).
\end{multline}
As a consequence,
\begin{equation}\label{eq:STreg}
 \lim_{n\to\infty} \inf_{v\in W_{\disc_n}} \disT_{\disc_n}(w,v) = 0.
\end{equation}
\end{lem}
\begin{proof}
Recalling the definition \eqref{eq:def.Suv} of $\disT_{\disc}$, the estimate \eqref{eq:errestreg} is a consequence of Lemmas \ref{lem:cvschemeex}, \ref{lem:interpinVex} and \ref{lem:interpinVpex}, once we notice that, for all $u\in \wunp$,
\[
 \Vert \agrad R u \Vert_{\lpd}\le \ctel{cst:GR.cont} \Vert u\Vert_\lp \le  \cter{cst:GR.cont} C_P \Vert u\Vert_\wunp,
\]
the first inequality being obtained by selecting $\xi=v=u$ in \eqref{eq:defR}, while the second follows from \eqref{eq:poincare}.

The relation \eqref{eq:STreg} follows from Lemmas \ref{lem:WT} and \eqref{lem:consistT}.

\end{proof}

\section{Numerical illustration}\label{sec:num}

\subsection{Irregular initial data}

One of the key features of the error estimate in Theorem \ref{thm:errest} is that it does not require any regularity on the solution beyond the one provided by the model itself. Let us apply our error estimate to a case where the continuous solution of the problem is such that $u'\notin L^2(0,T;\lp)$.
Let $\Omega = (0,1)$, $\lp = \lpd = L^2(\Omega)$, $\agrad u = \partial_x u$, $\wunp = H^1_0(\Omega)$, $\Phi = 0$, $\Lambda = {\rm Id}$, $f=0$, $\bF=0$, $\xi_0(x) = 1$, $T=1/10$.
Then the solution of Problem \eqref{eq:pbcont} is given, for $t\in (0,T]$ and $x\in [0,1]$, by 
\[
 u(t)(x) = \sum_{p\in\mathbb{N}} \frac 4 {(2p+1)\pi} \exp(-((2p+1)\pi)^2 t) \sin((2p+1)\pi x).
\]
We define $(X_\disc,\api_\disc,\agrad_\disc)$, letting $k = 0.9 h^2$, from the Control Volume Finite Element gradient discretisation \cite[8.4 p274]{gdm}. It consists, for a given $M\in\mathbb{N}^\star$, in defining  $h = 1/(M+1)$, $X_\disc = \mathbb{R}^M$, and, for any $w:=(w_i)_{i=1,\ldots,M}$, letting $w_0 = w_{M+1} = 0$,
\[
 \api_\disc w(x)= w_i\hbox{ if }x\in \left((i-\frac 1 2)h,(i+\frac 1 2)h\right)\cap(0,1),\ i=0,\ldots,M+1,
\]
\[
 \agrad_\disc w(x)= \frac {w_{i+1} - w_i} {h}\hbox{ if }x\in (ih,(i+1)h),\ i=0,\ldots,M.
\]
Let $\varphi_i:[0,1]\to [0,1]$, for $i=1,\ldots,M$, be the $P^1$ finite element basis function: $\varphi_i(jh) = \delta_{ij}$ for all $j=0,\ldots,M+1$, and $\varphi_i$ continuous and piecewise affine. Then, setting
\[
 v := \sum_{i=1}^M w_i \varphi_i \in H^1_0(\Omega),
\]
we get
\[
 \agrad_\disc w = \agrad v.
\]
The advantage of this discretisation method is that is satisfies monotonicity properties, due to the fact that the mass matrix is lumped, accounting for the definition of $\api_\disc$. We show in Figure \ref{fig:sol} the exact solution at different times, and the approximate solution obtained by the scheme at the final time.
\begin{figure}[!ht]
	\centering
 \includegraphics[scale=0.6]{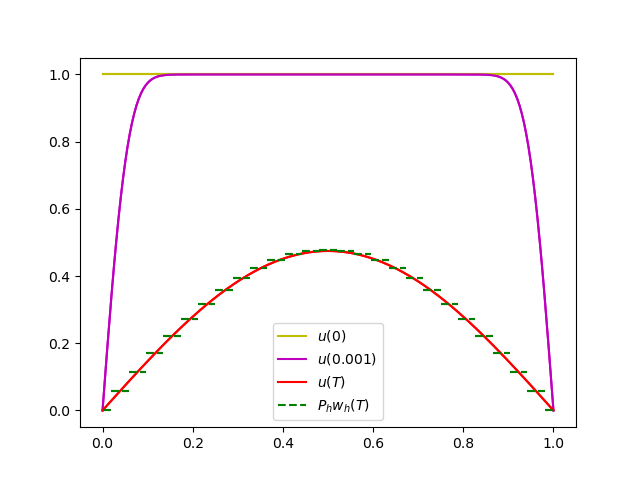}
 \caption{Exact solution at different times and approximate solution at the final time, irregular initial data} \label{fig:sol}
\end{figure}
In this case, the continuous solution does not satisfy $u'\in L^2(0,T;\lp)$, nor $u\in L^2(0,T;H^2(\Omega))$. Indeed, for any $T>0$, we have
\[
 \lim_{\varepsilon\to 0} \Vert u'\Vert_{L^2(\varepsilon,T;L^2(\Omega))} = \lim_{\varepsilon\to 0} \Vert \Delta u\Vert_{L^2(\varepsilon,T;L^2(\Omega))} = +\infty.
\]
This can be shown by noticing that
\[
  \Vert u'\Vert_{L^2(\varepsilon,T;L^2(\Omega))}^2 = \frac 1 2 (\Vert \agrad u(\varepsilon)\Vert_{\lpd}^2 - \Vert \agrad u(T)\Vert_{\lpd}^2). 
\]
If $ \Vert u'\Vert_{L^2(\varepsilon,T;L^2(\Omega))}$ were bounded as $\varepsilon\to 0$, so would be $\Vert \agrad u(\varepsilon)\Vert_{\lpd}$. Since $u(\varepsilon)\to \xi_0$ in $\lp$, this would imply that $\xi_0\in\wunp$, which does not hold.

\medskip

Computing the error terms involved in \eqref{eq:def.Suv} in Theorem \ref{thm:errest}, we remark that, since the right-hand-side vanishes,
\[
 \Vert \agrad Ru' - \agrad_\disc R_\disc\partial w\Vert_{L^2(\lpd)} = \Vert \agrad u - \agrad_\disc  w\Vert_{L^2(\lpd)}.
\]
It therefore suffices to evaluate 
\[
E_1 = \Vert \agrad u - \agrad_\disc  w\Vert_{L^2(\lpd)},\ E_2 = \max_{t\in[0,T]}\Vert u(t) - \api_\disc  w(t)\Vert_{\lp}.
\]
In order to compute an accurate value of $E_1$, we remark that
\[
 \int_{(m-1)k}^{mk} \Vert \agrad u(t) - \agrad_\disc w^{(m)}\Vert_{\lpd}^2 {\rm d}t = T_1^{(m)} - 2 T_2^{(m)} + T_3^{(m)},
\]
with
\[
 T_1^{(m)} = \int_{(m-1)k}^{mk} \Vert \agrad u(t)\Vert_{\lpd}^2 {\rm d}t = \frac 1 2 (\Vert u(mk)\Vert_{\lp}^2 - \Vert u((m-1)k)\Vert_{\lp}^2),
\]
\[
 T_2^{(m)} = \int_{(m-1)k}^{mk} \langle \agrad u(t),\agrad v^{(m)}\rangle_{\lpd} {\rm d}t = \langle  u(mk)- u((m-1)k), v^{(m)}\rangle_{\lp},
\]
and
\[
 T_3^{(m)} = \int_{(m-1)k}^{mk} \Vert \agrad v^{(m)}\Vert_{\lpd}^2 {\rm d}t.
\]
Hence
\[
 \int_{0}^{T} \Vert \agrad u(t) - \agrad_\disc v_\disc(t)\Vert_{\lpd}^2 {\rm d}t = \frac 1 2 (\Vert u(T)\Vert_{\lp}^2 - \Vert \xi_0\Vert_{\lp}^2) + \sum_{m=1}^N(T_2^{(m)}+T_3^{(m)}).
\]
It then suffices to compute the terms $ \langle  u(mk)- u((m-1)k), v^{(m)}\rangle_{\lp}$ using quadrature formulas.
%

We observe in Figure \ref{fig:etuCV} that $E_1$ and $E_2$ behave as $C\sqrt{h}$, which is the expected interpolation order for this case (note that the function $\bv$ given by \eqref{eq:defbv} is null in this case, which implies that $\limcT_\disc(\bv) = 0$).

\begin{figure}[!ht]
	\centering
 \includegraphics[scale=0.6]{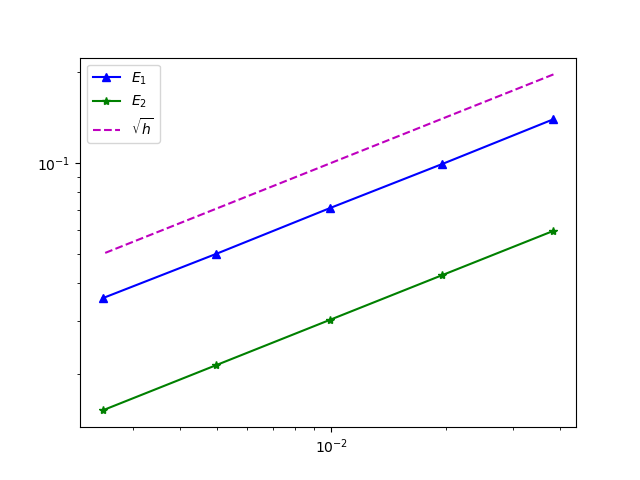}
 \caption{Errors $E_1$ and $E_2$ for different values of $h$, irregular initial data} \label{fig:etuCV}
\end{figure}

\subsection{Irregular right-hand-side}

We consider again $\Omega = (0,1)$, $\lp = \lpd = L^2(\Omega)$, $\agrad u = \partial_x u$, $\wunp = H^1_0(\Omega)$, $\Phi = 0$, $\Lambda = {\rm Id}$, $T=1/10$, and $\xi_0,f,\bF$ such that the solution of Problem \eqref{eq:pbcont} is given, for $t\in (0,T]$ and $x\in [0,1]$, by 
\[
 u(t)(x) = t \min(x,1-x).
\]
Hence we let $\xi_0 = 0$, $f(t)(x) = \min(x,1-x)$ and $\bF(t)(x) = -\partial_x u(t)(x)$.

We approximate this problem on $[0,T]$ using the same discretisation method as in the previous section with $k = 0.9 h^2$, and specifying odd values for $M$.  We show in Figure \ref{fig:sols} the exact solution at different times, and the approximate solution obtained by the scheme described below at the final time. 

\begin{figure}[!ht]
	\centering
 \includegraphics[scale=0.6]{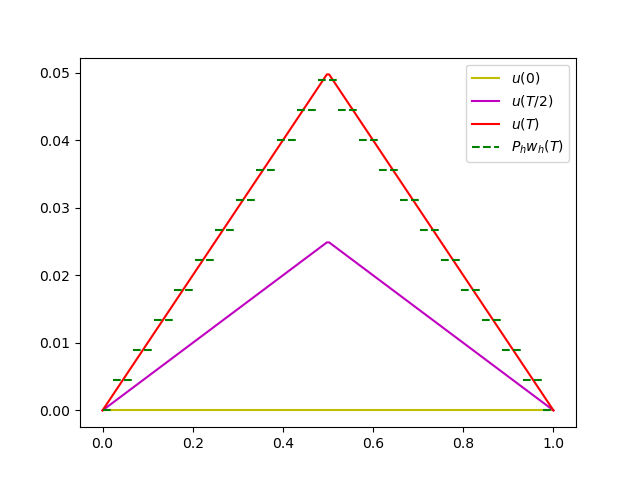}
 \caption{Exact solution at different times and approximate solution at the final time, irregular right-hand-side} \label{fig:sols}
\end{figure}

We obtain the following results ($E_1$ and $E_2$ are defined as in the previous section).
%

We see in Figure \ref{fig:etuCVs} that $E_1$ behaves as $h^2$ and $E_2$ as $h$.
The fact that $M$ is odd enables the interpolation error to behave as $h$. 

\begin{figure}[!ht]
	\centering
 \includegraphics[scale=0.6]{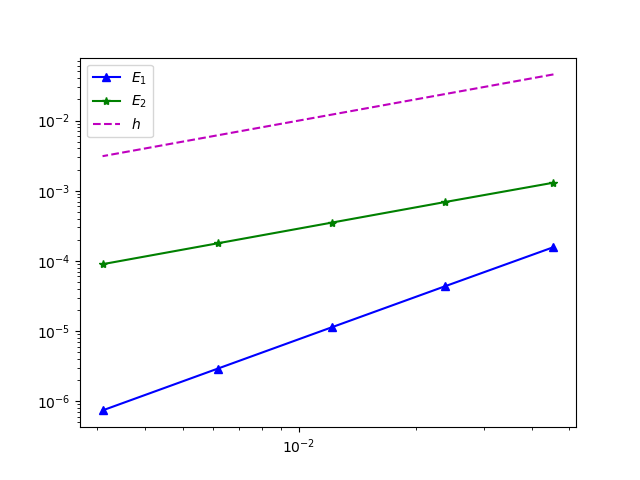}
 \caption{Errors $E_1$ and $E_2$ for different values of $h$, irregular right-hand-side} \label{fig:etuCVs}
\end{figure}

\bibliographystyle{abbrv}

\end{document}